\documentclass[12pt]{article}
\usepackage{amsfonts}
\usepackage{amsthm }
\usepackage{latexsym}
\usepackage{graphicx}
\usepackage{amsmath}
\usepackage{verbatim}
\graphicspath{{./figures/}}
\renewcommand{\epsilon}{\varepsilon}
\renewcommand{\rho}{\varrho}
\renewcommand{\phi}{\varphi}
\begin{document}
\newtheorem{definition}{Definition}[section]
\newtheorem{theorem}[definition]{Theorem}
\newtheorem{proposition}[definition]{Proposition}
\newtheorem{corollary}[definition]{Corollary}
\newtheorem{lemma}[definition]{Lemma}
\newtheorem{remark}[definition]{Remark}  
\newtheorem{conjecture}[definition]{Conjecture}
\newtheorem{problem}[definition]{Problem}
\newtheorem{example}[definition]{Example} 
\title{Cohomology in electromagnetic modeling}
\author{{\em Pawe\l{} D\l{}otko}\footnote{This Author's work was partially supported by MNiSW grant N N206 625439.} \\
        Institute of Computer Science \\
        Jagiellonian University \\
        ul.~St.~\L{}ojasiewicza~6 \\
        30-348~Krak\'ow, Poland
  \and
        {\em Ruben Specogna} \\
        Universit\`a di Udine,\\
        Dipartimento di Ingegneria Elettrica,\\
        Gestionale e Meccanica,\\
        Via delle Scienze 208,\\
        33100 Udine, Italy.\\
        ruben.specogna@uniud.it}
\maketitle
\newpage
\begin{abstract}
Electromagnetic modeling provides an interesting context to present a link between physical phenomena and homology and cohomology theories. Over the past twenty-five years, a considerable effort has been invested by the computational electromagnetics community to develop fast and general techniques for potential design. When magneto-quasi-static discrete formulations based on magnetic scalar potential are employed in problems which involve conductive regions with holes, \textit{cuts} are needed to make the boundary value problem well defined. While an intimate connection with homology theory has been quickly recognized, heuristic definitions of cuts are surprisingly still dominant in the literature.

The aim of this paper is first to survey several definitions of cuts together with their shortcomings. Then, cuts are defined as generators of the first cohomology group over integers of a finite CW-complex. This provably general definition has also the virtue of providing an automatic, general and efficient algorithm for the computation of cuts. Some counter-examples show that heuristic definitions of cuts should be abandoned. The use of cohomology theory is not an option but the invaluable tool expressly needed to solve this problem.
\\
{\em keywords: algebraic topology, (co)homology, computational electromagnetics, cuts}
\end{abstract}
%
%
\tableofcontents
\newpage

\section{Introduction}
There is a remarkable interest in the efficient numerical solution of large-scale three-dimensional electromagnetic problems by Computer-Aided Engineering (CAE) softwares which enables a rapid and cheap design of practical devices together with their optimization.

Electromagnetic phenomena are governed by Maxwell's laws \cite{maxwell} and constitutive relations of materials. This paper focuses on the numerical solution of magneto-quasi-static Boundary Value Problems (BVP)---also called eddy-current problems---which neglect the \textit{displacement current} in the Amp\`{e}re--Maxwell's equation \cite{maxwell}, \cite{carpenter}, \cite{bossavitcompel}. This well-studied class has quite a big number of industrial applications such as non-destructive testing, electromagnetic breaking, metal separation in waste, induction heating, metal detectors, medical imaging and hyperthermia cancer treatment.

The range of CAE applications is sometimes bounded by the high computational cost needed to obtain the solution, hence state-of-the-art numerical methods are usually sought. Recently, the Discrete Geometric Approach (DGA) gained popularity, becoming an attractive method to solve BVP arising in various physical theories, see for example \cite{kron}--
\cite{bossa_libro}.
The DGA bears strong similarities to compatible or mimetic discretizations \cite{mimetic}, \cite{compatible}, discrete exterior calculus \cite{dec} and finite element exterior calculus \cite{arnoldima}, \cite{arnoldacta}, \cite{arnold}.
All these methods present some pedagogical advantages with respect to the standard widely used Finite Element Method (FEM).

First of all, the topological nature of Maxwell's equations and the geometric structure behind them allows to reformulate the mathematical description of electromagnetism directly in algebraic form. Such a reformulation can be formalized in an elegant way by using algebraic topology \cite{branin}, \cite{tonti_classificazione}, \cite{mattiussi}, \cite{tonti_marsiglia}, \cite{compatible}, \cite{arnoldima}, \cite{arnoldacta}, \cite{arnold}. Taking advantage of this formalism, as illustrated in Section 3, physical variables are modeled as cochains and Maxwell's laws are enforced by means of the coboundary operator.
Information about the metric and the physical properties of the materials is encoded in the constitutive relations, that are modeled as discrete counterparts of the Hodge star operator \cite{mattiussi}, \cite{tarh}, \cite{dho_hiptmair}, \cite{compatible}, \cite{arnold} usually called \textit{constitutive matrices} \cite{tonti}.
By combining Maxwell's with constitutive matrices, an algebraic system of equations is directly obtained, yielding to a simple, accurate and efficient numerical technique.
The difference of the DGA with respect to similar methods lies in the computation of the constitutive matrices, which in the DGA framework is based on a closed-form geometric construction. For a computational domain discretized by using a geometric realization of a polyhedral cell complex, one may use the techniques described in \cite{cmamepoly}, \cite{jcppoly} and references therein, without losing the symmetry, positive-definiteness and consistency of the constitutive matrices which guarantee the convergence of the method. Hence, we consider the most general situation of dealing with a polyhedral cell complex.

Our purpose is not to present the widely known DGA or similar discretizations, but to use it as a working framework. This choice does not limit the generality of the results, since the standard Finite Element Method (FEM) and the Finite Differences (FD) can be easily reinterpreted in the DGA framework as in \cite{mattiussi}, \cite{Bossa98}, \cite{tarh}, \cite{gal_hodge}, \cite{marrone}, \cite{compatible}, \cite{arnoldima}, \cite{arnoldacta}, \cite{arnold}. Consequently our results can be extended, without any modification, to the corresponding widely used FEM formulation.

The paper is focused on a particular application of algebraic topology, namely the potential design for the efficient numerical solution of eddy-currents Boundary Value Problems (BVP).
Electromagnetic potentials are auxiliary quantities frequently used to enforce some of the Maxwell's laws implicitly.
There are two families of formulations for magneto-quasi-static problems, depending on the set of potentials chosen, see for example \cite{carpenter}, \cite{bossavitcompel}, \cite{bossa_libro}. To better understand the link between (co)homology theory and physics, our attention is focused on the $h$-oriented geometric formulations, namely the $T$-$\Omega$ \cite{sst}, \cite{cmame}, which are based on a magnetic scalar potential $\Omega$. Those are much more efficient than the complementary family of $b$-oriented $A$ and $A$-$\chi$ formulations \cite{spectre}, both in terms of memory requirements and simulation time. The main reason is that usually $h$-oriented formulations require about an order of magnitude less unknowns.
Nonetheless, when $h$-oriented formulations involve electrically conductive regions with holes (i.e., the first homology group of some conductor is non-trivial), the design of potentials is not straightforward. \textit{Cuts} are needed to be introduced to make the BVP well defined.
How to define cuts and devise an efficient and automatic algorithm to compute them has been an intellectual challenge for the computational electromagnetics community for the last twenty-five years. While a connection of this issue with homology theory was quickly recognized by Kotiuga \cite{k1} more than twenty years ago, heuristic definition of cuts based on intuition are surprisingly still dominant in the literature.

The aim of the paper is to rigorously present a systematic design of the potentials employed in $h$-oriented formulations by taking advantage of homology and cohomology theories. In particular, at the end of the presentation we are able to formally demonstrate that cuts are generators of the first cohomology group over integers of the insulating region.
The originality of the approach presented in this paper lies also in the fact that the design of potentials is tackled directly within a topological setting. In fact, thanks to the reformulation of Maxwell's laws by using the coboundary operator, homology and cohomology with integer coefficients are employed from the beginning for the potential design in place of the standard \textit{de Rham cohomology}, see for example \cite{bott}, routinely used in the FEM context, see for example \cite{ren}, \cite{k1}, \cite{GrossKotiuga}, \cite{henrotte}. In the FEM framework, the so-called \textit{non-local basis functions} are added to the set of usual basis functions to be able to span the de Rham first cohomology group, see for example \cite{ren}, \cite{henrotte}, \cite{dular_any}, \cite{dular_coupling}, \cite{dular}, \cite{dular_h}, \cite{bossavit_mostgeneral}, \cite{dular_jcam}. Moreover, employing the DGA, new insights into the formulation can be presented by exploiting the dualities arising when, as in the DGA framework, two interlocked cell complexes---one dual of the other---are employed. For example, the physical interpretation of the non-local basis functions as non-local Faraday's equations will become apparent.

The second purpose of the paper is to present a survey on definitions of cuts already presented in the literature showing their shortcomings. Concrete counter-examples show why heuristic definitions should be abandoned and that cohomology is not one of the possible options but something which is expressly needed to face this problem.

The paper is structured as follows.
In Section~\ref{sec:ATinCP}, a survey on the relevant topics of algebraic topology is provided together with a link to electromagnetic modeling. In Section~\ref{sec:MaxwellAlgForm}, Maxwell's laws casted in algebraic form are recalled.
Section~\ref{sec:TowardADefinitionOCuts} shows the problem related with expressing Amp\`{e}re's law with a magnetic scalar potential when dealing with conducting regions with holes. Then, to solve this issue, \textit{cuts} are defined as generators of the first cohomology group over integers of the insulating region. In Section~\ref{sec:F-Omega}, the $T$-$\Omega$ geometric formulation to solve magneto-quasi-static BVP is described.
Section~\ref{coh} contains a survey of the definitions of cuts presented in the literature together with an illustration of their shortcomings. In Section~\ref{sec:numericalEx}, a short discussion on how to compute the cohomology generators is presented. Finally, in Section~\ref{sec:conclusions}, the conclusions are drawn.

\section{Algebraic topology in computational physics}
\label{sec:ATinCP}

\subsection{Basic concepts in algebraic topology}
\label{sec:BasicAT}
In this Section, some basic concepts of algebraic topology are reviewed.
Let us first introduce the concept of \textit{finite regular CW-complex}.
An \textit{n-cell} $e^n$ is an open subset of a Hausdorff space~$X$ homeomorphic to the $n$-dimensional unit ball~$B_1^n(0) \subset \mathbb{R}^n$.
An $n$-cell~$e^n$ is said to be \textit{attached} to the closed subset $K \subset X$ if there exists a continuous map $f : \overline{B_1^n(0)} \to \overline{e^n}$ such that~$f$ maps the open
ball~$B_1^n(0)$ homeomorphically onto~$e^n$ and $f(\partial B_1^n(0)) \subset K$ in a way that $e^n \cap K = \emptyset$.
 The map $f$ is referred to as \textit{characteristic map}. The finite CW-complexes are defined in the following way:
\begin{definition} \label{defcw}
Let~$X$ denote a Hausdorff space. A closed subset~$K \subset X$ is
called a {\em (finite) CW-complex\/} of dimension~$N$, if there exists an
ascending sequence of closed subspaces $K^0 \subset K^1 \subset \ldots
\subset K^N = K$ such that the following holds.
\begin{itemize}
\item[(i)] $K^0$ is a finite space.
\item[(ii)] For $n \in \{1,\ldots,N\}$, the set~$K^n$ is obtained from~$K^{n-1}$
by attaching a finite collection $K_n$ of $n$-cells.
\end{itemize}
\end{definition}
In this case, the subset~$K^n$ is called the {\em $n$-skeleton\/} and
the elements of~$K^0$ are called the {\em vertices\/} of~$K$.
An $N$-dimensional CW-complex is called {\em regular\/} if for
each cell~$e^n$, where $n \in \{1,\ldots,N\}$, there exists a characteristic map
$f : \overline{B_1^n(0)} \to \overline{e^n}$ which is a homeomorphism on $\overline{B_1^n(0)}$.
In this case, we say that the $m$-cell~$e^m$ is a {\em face\/} of an
$n$-cell~$e^n$, if the inclusion $e^m \subset \overline{e^n}$ holds, see~\cite{cwHomology}. Moreover, since our aim is to model physical objects, we restrict to the case of regular CW-complexes embedded in $\mathbb{R}^3$.
Finally, by a \textit{polyhedral mesh} we mean a cellular decomposition of the considered space, which is a regular CW-complex such that each cell of the complex is a polyhedron. In the paper we are going to use the terms mesh and regular CW-complex interchangeably.

Let $\mathcal{K}$ be a collection of cells of the regular CW-complex $K$. Let $\kappa : \mathcal{K}_{i-1} \times \mathcal{K}_i \rightarrow \{ -1 , 0 , 1 \}$ for $i \in \mathbb{Z}$ be the so-called {\em incidence index} which assigns to a pair of cells their incidence number (for further details consult~\cite{cwHomology}). Let $G$ denote the module of integers ($\mathbb{Z}$), real ($\mathbb{R}$) or complex ($\mathbb{C}$) numbers. The group of formal sums $\sum_{e \in \mathcal{K}_i} \alpha_e e$, where $\alpha_e \in G$ for every $e \in \mathcal{K}_i$, is the group of $i$-{\em chains} of the complex $\mathcal{K}$ and is denoted by $C_i( \mathcal{K} , G )$. 
For a chain $c = \sum_{e \in \mathcal{K}_i} \alpha_e e$ the support $|c|$ of $c$ consist of all elements $e \in \mathcal{K}_i$ such that $\alpha_e \not = 0$.
For two chains $c = \sum_{e \in \mathcal{K}_i} \alpha_e e$ and $d = \sum_{e \in \mathcal{K}_i} \beta_e e$ their scalar product is $\langle c , d \rangle = \sum_{e \in \mathcal{K}_i} \alpha_e \beta_e$.
The group of cochains $C^i( \mathcal{K} , G )$ is formally defined as the group of maps from elements of $C_i( \mathcal{K} , G )$ to $G$ with coordinatewise addition. However, it is possible and convenient for the computations to represent a cochain as a chain. Namely, to determine the value of a map $c^* : C_i( \mathcal{K} , G ) \rightarrow G$ on any $i$-chain, it suffices to know the value of $c^*$ on every $e \in \mathcal{K}_i$. In this way, it is possible to associate to the cochain $c^*$ a chain $c$ such that for any other chain $d \in C_i( \mathcal{K} , G )$ the value of cochain $c^*$ on chain $d$ is equal to $\langle c ,d \rangle$.
For a cochain $c^*$ its support $|c^*|$ consists of all the cells whose value of $c^*$ is nonzero.

In this paper, two kind of cochains are considered. The first are the integer-valued cochains---for example, the representatives of the first cohomology group generators over integers. The second are the complex-valued cochains, which model physical variables in the proposed application as discussed in Section \ref{sec:ComplexCochainsGeoEntities}.

Let us define the boundary map $\partial_i : C_i( \mathcal{K} , G ) \rightarrow C_{i-1}( \mathcal{K} , G )$. For an element $e \in \mathcal{K}_i$ we define
\[ \partial_i e = \sum_{f \in \mathcal{K}_{i-1}} \kappa( f , e ) f \]
and extend it linearly to the map from $C_i( \mathcal{K} , G )$ to $C_{i-1}( \mathcal{K} , G )$.
The coboundary map $\delta^i : C^i( \mathcal{K} , G ) \rightarrow C^{i+1}( \mathcal{K} , G )$ is defined for $e \in \mathcal{K}_i$ by
\[ \delta^i e = \sum_{f \in \mathcal{K}_{i+1}} \kappa(e,f) f \]
and extended linearly to the map from $C^i( \mathcal{K} , G )$ to $C^{i+1}( \mathcal{K} , G )$.
It is standard that $\partial_{i-1} \partial_{i} = \delta^{i} \delta^{i-1} = 0$ for every $i \in \mathbb{Z}$, see~\cite{Hatcher}. Moreover, the coboundary map is dual to boundary map in homology. In fact, it can be equivalently defined with the equality $\langle \delta c^* , d \rangle = \langle c^* , \partial d \rangle$ for every $c^* \in C^{i-1}(\mathcal{K},G)$ and for every $d \in C_i(\mathcal{K},G)$, see~\cite{Hatcher}.

The boundary operator gives rise to a classification of chains. The group of $i$-{\em cycles} is $Z_i( \mathcal{K} , G ) = \{ c \in C_i( \mathcal{K} , G ) | \partial c = 0 \}$. The group of $i$-{\em boundaries} is $B_i( \mathcal{K} , G ) = \{ c \in C_i( \mathcal{K} , G ) | \text{ there exist } d \in C_{i+1}( \mathcal{K} , G ) | \partial d = c \}$. Intuitively, a cycle is a chain whose boundary vanishes while a boundary is a cycle which can be obtained as the boundary of some higher dimensional chain. The $i$th {\em homology group} is the quotient $H_i( \mathcal{K} , G ) = Z_i(\mathcal{K} , G) / B_i( \mathcal{K} , G )$. The cycles that are not boundaries are nonzero in $H_i( \mathcal{K} , G )$. The cycles that differ a by a boundary are in the same equivalence class. Given a chain $c$, by $[c]$ we denote its homology class, i.e. the class containing all the cycles homologous to $c$. By generators of the $i$th homology group we mean a minimal set of classes which generates $H_i( \mathcal{K} , G )$. In the following, for the sake of brevity, by generators we also mean the cycles being representatives of the considered classes that generate $H_i( \mathcal{K} , G )$.

Dually, with the coboundary operator the cochains may be classified. The group of $i$-{\em cocycles} is $Z^i( \mathcal{K} , G ) = \{ c \in C^i( \mathcal{K} , G ) | \delta c = 0 \}$. The group of $i$-{\em coboundaries} is $B^i( \mathcal{K} , G ) = \{ c \in C^i( \mathcal{K} , G ) | \text{ there exist } d \in C^{i-1}( \mathcal{K} , G ) | \delta d = c \}$. The $i$th {\em cohomology group} is the quotient $H^i( \mathcal{K} , G ) = Z^i(\mathcal{K} , G) / B^i( \mathcal{K} , G )$. By generators of the $i$th cohomology group we mean a minimal set of classes which generates $H^i( \mathcal{K} , G )$. Also in this case, for the sake of brevity, by generators we also mean the cocycles being representatives of the considered classes that generate $H^i( \mathcal{K} , G )$.

In the following, we will use also the standard concept of the so-called \emph{relative} (co)homology. In relative (co)homology, some parts of the complex may be considered irrelevant. Let $\mathcal{K}$ be the considered regular CW-complex and $\mathcal{S} \subset \mathcal{K}$ be a closed sub-complex of $\mathcal{K}$. The concept of relative homology bases on the definition of {\em relative chains} $C_i( \mathcal{K} , \mathcal{S} , G ) = C_i( \mathcal{K} , G ) / C_i( \mathcal{S} , G )$. The definition of relative cycles $Z_i( \mathcal{K} , \mathcal{S} , G ) $, relative boundaries $B_i( \mathcal{K} , \mathcal{S} , G ) $  and relative homology group $H_i( \mathcal{K} , \mathcal{S} , G ) $ remain unchanged with respect to the absolute version once relative chains are used. 
Exactly the same approach is also employed in defining relative cohomology.

In theorem \ref{th:homologyAreTorsionFree}, it will be recalled that there is no \emph{torsion}~\cite{Hatcher} in the homology and cohomology groups dealing with regular CW-complexes embedded in $\mathbb{R}^3$.
A direct consequence of the Universal Coefficient Theorem for cohomology, see~\cite{Hatcher}, is that in the considered torsion-free case the generators of the cohomology group over integers and the generators of the cohomology group over complex numbers are in a bijective correspondence (for further details see~\cite{tc}). Therefore, all the computations are rigorously performed by using integer arithmetic and the obtained cohomology generators are valid cohomology generators also in the case of complex coefficients.
The theory of (co)homology computations for regular CW-complexes can be found in \cite{cwHomology}. The important point is that for (co)homology computations only the incidence indexes  $\kappa$ between cells of $K$ are needed. This fact provides an easy way of representing CW-complexes for the (co)homology computations with a computer by using a pointer-based data structure. Once the theory is provided, any of the existing libraries like~\cite{capd} or~\cite{chomp} used to compute homology of, for instance, cubical sets can be adopted for cohomology computation of an arbitrary regular CW-complex.


Let us now introduce the concept of \textit{exact sequences}.
Let $A_1,\ldots,A_{m+1}$ be abelian groups and let $\alpha_i : A_i \rightarrow A_{i+1}$ for $i \in \{1,\ldots,m\}$ be homomorphisms between them. The sequence
$$A_1 \xrightarrow{\alpha_1} A_2 \xrightarrow{\alpha_2} \ldots \xrightarrow{\alpha_{m-1}} A_{m} \xrightarrow{\alpha_m} A_{m+1}$$
is called an exact sequence if $\mathrm{Im}(\alpha_{i}) = \mathrm{Ker}(\alpha_{i+1})$ for every $i \in \{1,\ldots,m-1\}$. For $m=4$ and $A_1 = A_4 = 0$ the exact sequence
$$0 \rightarrow A_2 \xrightarrow{\alpha_2} A_{3} \xrightarrow{\alpha_3} A_{4} \rightarrow 0$$
is referred to as \textit{short exact sequence}.
The so-called exact sequence of the reduced homology provides us a tool to relate the homology group of the space $X$, its subspace $A$ and the relative homology of the pair $(X,A)$.
\begin{theorem}[see Theorem 2.16, \cite{Hatcher}]
\label{th:relHomologyExactSequence}
If $X$ is a regular CW-complex and $A \subset X$ is a sub-complex of $X$, then there is a long exact sequence
$$\ldots \xrightarrow{\partial} H_n(A) \xrightarrow{i_{*}} H_n(X) \xrightarrow{j_{*}} H_n(X,A) \xrightarrow{\partial} H_{n-1}(A) \xrightarrow{i_{*}} \ldots.$$
\end{theorem}
The map $\partial : H_n(X,A) \rightarrow H_{n-1}(A)$ maps a class $[\alpha] \in H_n(X,A)$ to a class $[\partial \alpha] \in H_{n-1}(A)$.
It is straightforward that, when $H_n(X)$ is trivial, from the exactness of the sequence, $\partial : H_n(X,A) \rightarrow H_{n-1}(A)$ is an isomorphism.

To state, in Section~\ref{sec:TowardADefinitionOCuts}, the dualities between first homology and first cohomology groups of subsets of $\mathbb{R}^3$ the following theorems are required:
\begin{theorem}[\cite{Hatcher}, Proposition A.4, Corollary 3.44]
\label{th:homologyAreTorsionFree} If $X \subset \mathbb{R}^n$ is a finite CW-complex, then $H_i(X,\mathbb{Z})$ is $0$ for $i \geq n$ and torsion free for $i=n-1$ and $i=n-2$.
\end{theorem}
For $n=3$ the Theorem~\ref{th:homologyAreTorsionFree} states that the first and the second homology group of a finite CW-complexes embeddable in $\mathbb{R}^3$ are torsion free.

For CW-complexes the following---stronger from standard---version of excision theorem holds:
\begin{theorem}[Corollary 2.24, \cite{Hatcher}]
\label{th:HatcherCorollary224}
If the CW-complex $X$ is the union of sub-complexes $A$ and $B$, then the inclusion $(B,A \cap B)\hookrightarrow (X,A)$ induces an isomorphism $H_n(B,A \cap B) \rightarrow H_n(X,A)$ for all $n$.
\end{theorem}

\subsection{Dual chain complex}
\label{sec:ATinCE}
A polyhedral mesh is used to model the domain of interest of the electromagnetic problem. Let us fix the polyhedral mesh $\mathcal{K} = \{K_n\}_{n \in \mathbb{N}}$. Let us now define the dual mesh $\mathcal{B}$. The construction is a straightforward extension of the construction of the dual mesh for a simplicial complex explained in Figure~\ref{fig:barysubd}, see \cite{Munkres}.
\begin{figure} [!h]
\centering
\includegraphics[width=12cm]{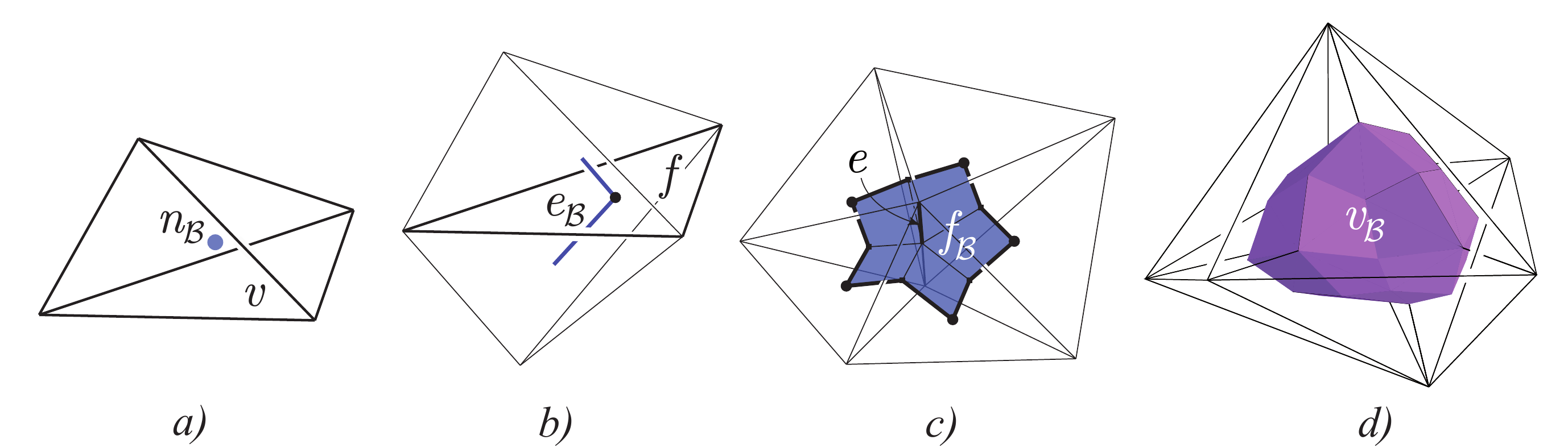}
\caption{a) A cell $v$ of a simplicial complex and its one-to-one node of the barycentric complex $n_\mathcal{B}$; b) A face $f$ of the simplicial complex and its one-to-one edge of the barycentric complex $e_\mathcal{B}$; c) An edge $e$ of the simplicial complex and its one-to-one face of the barycentric complex $f_\mathcal{B}$; d) One-to-one correspondence between a node $n$ of the simplicial complex and the volume of the barycentric complex $v_\mathcal{B}$.} \label{fig:barysubd}
\end{figure}

Let $n_0 = dim\ \mathcal{K}$ be the dimension of the complex $\mathcal{K}$.
For every cell $c \in K_i$ for $i \in \{0,\ldots,n_0\}$ by $\tilde{c} \in B_{n_0 - i}$ let us denote the corresponding element in dual mesh $\mathcal{B}$.
For every $c \in K_{n_0}$, the corresponding element $\tilde{c} \in \mathcal{B}_{0}$ is simply the barycenter of $c$. The remaining cells of $\mathcal{B}$ are defined recursively in the following way. For $c \in K_i$ for $i \in \{0,\ldots,n_0-1\}$, let $\{c_1,\ldots,c_n\} = |\delta c|$ and let $B(c)$ denotes the barycenter of $c$. Then $\tilde{c} = \bigcup_{i=1}^n \bigcup_{x \in \tilde{c}_i} \{[x,B(c)]\}$, where $[x,y]$ denotes the line segment joining $x$ and $y$. In the paper, we denote by $\tilde{\partial}$ and $\tilde{\delta}$ the boundary and coboundary operator in the dual complex $\mathcal{B}$.
To let the chain complex of $\mathcal{B}$ be dual---in purely algebraic sense---to the chain complex of $\mathcal{K}$, the boundary operator is defined as follows:
\[ \langle \partial c , d \rangle_{\mathcal{K}} = \langle \tilde{\partial} \tilde{d} , \tilde{c} \rangle_{\mathcal{B}}\ \  \forall c \in C_i(\mathcal{K}), d \in C_{i-1}(\mathcal{K}) \text{ for } i \in \mathbb{N}.\]
We would like to point out that the presented construction is valid only in the case of manifolds without boundary. In case of the presence of a boundary, first the complex dual to the boundary is constructed. Then, the two dual complexes, namely the one described above and the complex dual to the boundary, are merged in an obvious way. Nonetheless, the construction of the dual complex is fundamental only to develop the formulation while, for the computations, the explicit construction of the dual complex can be avoided.
The historical context of the idea of barycentric dual complex is mentioned in Section~\ref{sec:cutsOnDualCmplx}.

\subsection{Physical variables as complex-valued cochains}
\label{sec:ComplexCochainsGeoEntities}
Let $\mathcal{K}$ be a homologically trivial polyhedral cell complex in $\mathbb{R}^3$.
Let $\mathcal{K}_c$ be a closed sub-complex of $\mathcal{K}$ which models the conducting region. The $\mathcal{K} \setminus int\ \mathcal{K}_c$ is an insulating region the closure of which is meshed by $\mathcal{K}_a$, which is a closed sub-complex of $\mathcal{K}$. The interface of the conducting and insulating region is meshed by $\mathcal{K}_c \cap \mathcal{K}_a$. Moreover, it is assumed that $\mathcal{K}_c$ and $\mathcal{K}_a$ are non-empty. The dual sub-complexes are denoted by $\mathcal{B}_a$ and $\mathcal{B}_c$, respectively.

Since our aim is to solve the eddy-current problem in the frequency domain, the \textit{physical variables} are modeled in this paper as complex-valued cochains\footnote{A frequency domain eddy-current problem implies that all physical variables exhibit a time variation as isofrequential sinusoids. By using the standard \textit{symbolic method}, see \cite{steinmetz}, constant complex numbers called \textit{phasors} are used to represent the sinusoids. If the eddy-current problem has to be solved in time domain, the reals should be used in place of the complex numbers through the paper without any further modification.}. According to Tonti's classification of physical variables \cite{tonti_classificazione}, \cite{tonti_marsiglia}, \cite{tonti}, there is a unique association between a physical variable, such as electric current or magnetic flux \cite{maxwell}, \cite{tonti_classificazione}, \cite{tonti_marsiglia}, and an oriented geometric element of the two cell complexes $\mathcal{K}$ and $\mathcal{B}$. The cochain values, usually called Degrees of Freedom (DoFs) in computational physics (see for example \cite{bossa_libro}), have a direct physical interpretation: By using the so-called \textit{de Rham map} \cite{dodziuk}, they are defined as integrals of the electromagnetic differential forms over the elements of the complex\footnote{For example, the magneto-motive force (m.m.f.) DoF relative to the $1$-dimensional cell $e$ is the integral of the differential $1$-form \textit{magnetic field} over $e$.}.

\begin{figure} [!h]
\centering \includegraphics[width=12cm]{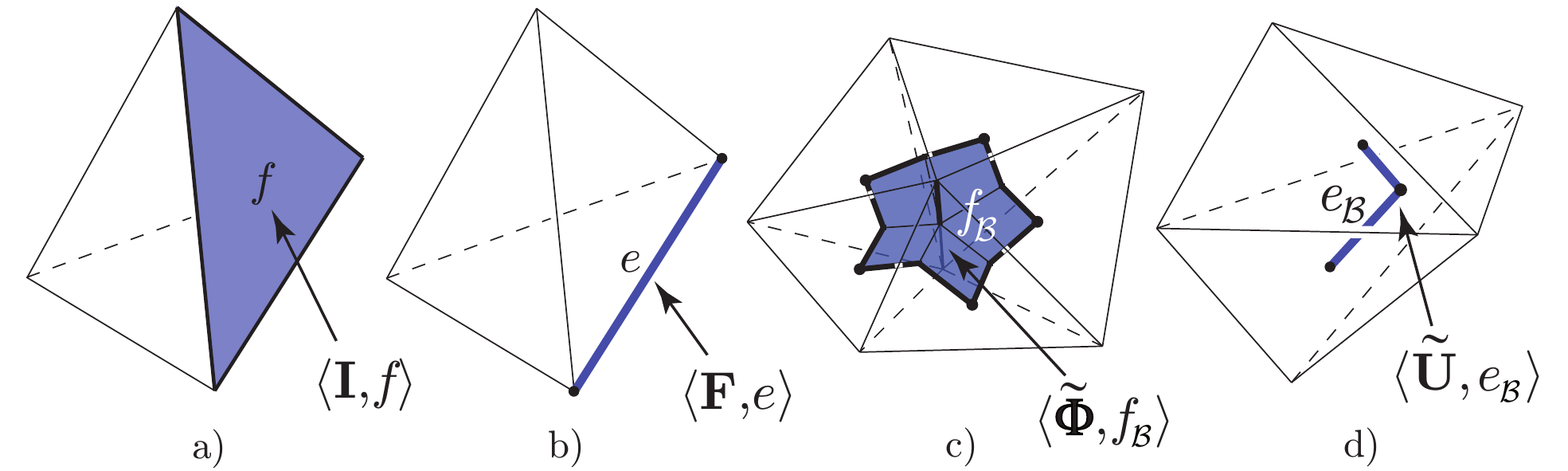}
\caption{Association of the degrees of freedom to the oriented geometrical entities.} \label{fig:association}
\end{figure}

The focus of this paper is on $h$-oriented formulations \cite{sst}, \cite{cmame}, so the following physical variables and related association with geometrical elements of $\mathcal{K}$ or $\mathcal{B}$ are considered:
\begin{itemize}
  \item $\langle \mathbf{I}, f \rangle$ is the \textit{electric current} associated to the face $f \in \mathcal{K}$, see Fig. \ref{fig:association}a. $\langle \mathbf{I}, f \rangle=0$ over the faces $f \in \mathcal{K}_a$ (with this definition, the current associated with the faces $f \in \mathcal{K}_a \cap \mathcal{K}_c$ is set to zero, since there is the need of a boundary condition that prevents the current to flow thought the boundary of the conductive region);
  \item $\langle \mathbf{F}, e \rangle$ is the \textit{magneto-motive force} (m.m.f.) associated to the edge $e \in \mathcal{K}$, see Fig. \ref{fig:association}b;
  \item $\langle \tilde{\boldsymbol{\Phi}}, f_{\mathcal{B}} \rangle$ is the \textit{magnetic flux} associated to the dual face $f_{\mathcal{B}} \in \mathcal{B}$, see Fig. \ref{fig:association}c;
  \item $\langle \tilde{\mathbf{U}}, e_{\mathcal{B}} \rangle$ is the \textit{electro-motive force} (e.m.f.) associated to the dual edge $e_{\mathcal{B}} \in \mathcal{B}$, see Fig. \ref{fig:association}d.
\end{itemize}
All these complex values, one for each $1$- or $2$-dimensional cell in the corresponding complex, are the coefficients of the corresponding complex-valued cochains denoted in boldface type.
For a fixed chain and cochain basis, each cochain can be represented as a vector which is used in the computations. In the following, since no confusion can arise, the notations for (co)chains and vectors representing them will be the same.

\section{Maxwell's equations in algebraic form and potentials analysis}
\label{sec:MaxwellAlgForm}
In this Section, the algebraic Maxwell's laws \cite{branin}, \cite{tonti_classificazione}, \cite{weiland} are reviewed.

The discrete \textit{current continuity law} enforces the dot product of the currents associated with faces belonging to the boundary of a volume $v \in \mathcal{K}_c$, with $dim\ v= 3$, to be zero 
\begin{equation}\label{continuity}
\langle \mathbf{I}, \partial v \rangle=\langle \delta \mathbf{I}, v \rangle=0,\,\, \forall v \in \mathcal{K}_c.
\end{equation}

Focusing on the generic face $f$, 
the discrete \textit{Amp\`{e}re's law} enforces the dot product of the m.m.f. $\mathbf{F}$ on the boundary of the face $f$ to match the current associated with $f$,
\begin{equation}\label{ampere}
\langle \mathbf{F}, \partial f \rangle =\langle \delta \mathbf{F}, f \rangle =\langle \mathbf{I}, f \rangle, \,\, \forall f \in \mathcal{K}.\\
\end{equation}
Since $\langle \mathbf{I}, f \rangle=0, \forall f \in \mathcal{K}_a$, $\mathbf{F}$ is a $1$-cocycle in $\mathcal{K}_a$ (however, it is not a cocycle in $\mathcal{K}$).

The discrete \textit{magnetic Gauss's law} enforces the dot product of the magnetic fluxes associated with the dual faces belonging to the boundary of a dual volume $v_{\mathcal{B}}$ to be zero
\begin{equation}\label{gauss}
\langle \tilde{\boldsymbol{\Phi}}, \tilde{\partial} v_{\mathcal{B}} \rangle=\langle \tilde{\delta} \tilde{\boldsymbol{\Phi}}, v_{\mathcal{B}} \rangle=0,\,\, \forall v_{\mathcal{B}} \in \mathcal{B}.
\end{equation}

Focusing on a dual face $f_{\mathcal{B}}$, the discrete \textit{Faraday's law} enforces the dot product of the e.m.f. $\tilde{\mathbf{U}}$ on the boundary of $f_{\mathcal{B}}$ to match the opposite of the variation of the magnetic flux through the face $f_{\mathcal{B}}$.
Considering problems in frequency domain, this translates in
\begin{equation}\label{faraday}
\langle \tilde{\mathbf{U}}, \tilde{\partial} f_{\mathcal{B}} \rangle =\langle \tilde{\delta} \tilde{\mathbf{U}}, f_{\mathcal{B}} \rangle = \langle - i \omega \tilde{\boldsymbol{\Phi}}, f_{\mathcal{B}} \rangle, \,\, \forall f_{\mathcal{B}} \in \mathcal{B},
\end{equation}
where $\omega$ is the angular frequency of the sinusoids equal to $2 \pi$ times the considered frequency.



The expressions just discussed of the four algebraic laws are called `local'. There exist also the so-called `non-local' versions of each of them, which are obtained by considering the balance not on exactly one geometrical entity but on a chain. It is straightforward to see that if the region is homologically trivial each non-local law can be obtained by considering a linear combination of local laws. Therefore, in this case, the non local laws do not bring any new information. As will be discussed in the next Sections, this does not hold for homologically non-trivial regions.

\subsection{A preliminary definition of potentials}\label{spre}
In this Section a preliminary definition of potentials employed in the formulation is presented and analyzed by means of algebraic topology.
For this preliminary definition let us assume that the considered complex $\mathcal{K}_c$ is homologically trivial. In particular, we use the fact that, when the conductors and the whole domain are homologically trivial cell complexes then, from a standard result on exact sequences, one has that the homology of the insulating domain is also trivial. How to generalize the definition of potentials in case of homologically non-trivial regions is the subject of Section \ref{sect_inc}.

Let us first analyze the potentials employed in the insulating region $\mathcal{K}_a$. It has been already indicated that $\mathbf{F}$ is a $1$-cocycle in $\mathcal{K}_a$, hence $\langle \delta \mathbf{F},f\rangle=0$ holds $\forall f \in \mathcal{K}_a$. Thus, since $\mathcal{K}_a$ is homologically trivial, a \textit{magnetic scalar potential} $\boldsymbol{\Omega}$ $0$-cochain can be introduced in the insulating region such that
\begin{equation}\label{defomega}
\langle \delta \boldsymbol{\Omega},e\rangle  = \langle\mathbf{F},e\rangle, \forall e \in \mathcal{K}_a.
\end{equation}

Let us analyze now the potentials employed in the conducting region $\mathcal{K}_c$. From (\ref{continuity}), we know that $\mathbf{I}$ is a $2$-cocycle in $\mathcal{K}_c$, hence $\langle \delta \mathbf{I},v\rangle=0$ holds $\forall v \in \mathcal{K}_c$.
Thus, an \textit{electric vector potential} $\mathbf{T}$ $1$-cochain can be introduced in the conducting region such that
\begin{equation}\label{deft}
\langle \mathbf{T},\partial f\rangle =\langle \delta \mathbf{T},f\rangle= \langle\mathbf{I},f\rangle, \forall f \in \mathcal{K}_c.
\end{equation}
Thanks to (\ref{ampere}), the following holds
\begin{equation}\label{deft2}
\langle \delta \mathbf{T},f\rangle = \langle\mathbf{I},f\rangle=\langle \delta \mathbf{F},f\rangle, \forall f \in \mathcal{K}_c.
\end{equation}
Since $\mathbf{F}$ and $\mathbf{T}$ are $1$-cochains such that $\delta \mathbf{F} = \delta \mathbf{T}$, they differ by a $0$-coboundary of a $0$-cochain $\mathbf{O}$, i.e. $\mathbf{F} = \mathbf{T} + \delta \mathbf{O}$. Since it is required that the magnetic scalar potential is continuous inside $\mathcal{K}$, we can extend the support of $\boldsymbol{\Omega}$ also inside $\mathcal{K}_c$ in such a way that $\langle \boldsymbol{\Omega} , n \rangle = \langle \mathbf{O} , n \rangle$ for every node $n \in \mathcal{K}_c$. We want to remark, that those extensions are valid also in case of homologically non-trivial complexes $\mathcal{K}_c$ and $\mathcal{K}_a$ and they are used further in the paper.
For brevity, let us define also the cochain $\mathbf{T}$ as cochain in $\mathcal{K}$. To this aim, we assume that $\langle \mathbf{T}, e \rangle = 0$ for every edge $e \in \mathcal{K}_a$.

To sum up, by using the potentials $\mathbf{T}$ and $\boldsymbol{\Omega}$, Amp\`{e}re's law (\ref{ampere}) and current continuity law (\ref{continuity}) can be enforced implicitly by considering the following expression for $\mathbf{F}$
\begin{equation}\label{potential2}
\mathbf{F}=\delta \boldsymbol{\Omega}+\mathbf{T}.
\end{equation}

Then, it is easy to show that employing this definition of potentials Amp\`{e}re's law holds in $\mathcal{K}_a$ (in fact, $\langle \delta \mathbf{F},f\rangle=\langle \delta \delta \boldsymbol{\Omega},f\rangle=\langle \mathbf{0},f\rangle=0=\langle \mathbf{I},f\rangle$, $\forall f \in \mathcal{K}_a$, since the current is zero in $\mathcal{K}_a$) and the current continuity law holds in all $\mathcal{K}$ (in fact, $\delta \mathbf{I}=\delta\delta \mathbf{T}=\mathbf{0}$). The remaining laws will be enforced by a system of equations in Section~\ref{sec:AlgebraicEquations}.

\noindent
\section{Potentials design}
\label{sec:TowardADefinitionOCuts}
\noindent
\subsection{Non-local Amp\`{e}re's law in homologically non-trivial domains}
\label{sect_inc}
Let us now remove the hypothesis of $\mathcal{K}_c$ being homologically trivial.
Amp\`{e}re's law can be written on a $1$-cycle $c \in Z_1(\mathcal{K}_a)$ and a $2$-chain $s \in C_2(\mathcal{K})$ such that $\partial s=c$:
\begin{equation}\label{nl_ampere}
\langle \mathbf{F},c \rangle=\langle \mathbf{I},s \rangle.
\end{equation}
Equation (\ref{nl_ampere}) is an example of a non-local equation, since the algebraic constraint is not enforced on a geometric element, like (\ref{ampere}), but involves geometric elements belonging to a wider collection, in this case the support of $s$ and its boundary $c$. In (\ref{nl_ampere}) we have not specified which $2$-chain $s$ has to be used for taking the dot product at the right-hand side (in fact, $s$ has been determined only up to its boundary $\partial s = c$). To solve this issue, we show with next Theorem that the value $\langle \mathbf{I} , s \rangle$ depends only on $\partial s$ and, therefore, (\ref{nl_ampere}) is well defined.
\begin{theorem}\label{thlc}
Let $s_1$ and $s_2$ be two $2$-chains such that $\partial s_1 = \partial s_2 = c$, where $c \in Z_1(\mathcal{K}_a)$. Then $\langle \mathbf{I} , s_1 \rangle = \langle \mathbf{I} , s_2 \rangle$.
\end{theorem}
\begin{proof}
From the assumptions, $s_1 - s_2 \in Z_2(\mathcal{K})$. Since $\mathcal{K}$ is homologically trivial, there exists $b \in C_3(\mathcal{K})$ such that $\partial b = s_1-s_2$. Consequently, $s_1=s_2+\partial b$ holds. Then, $\langle \mathbf{I},s_1 \rangle=\langle \mathbf{I},s_2 \rangle + \langle \mathbf{I},\partial b \rangle = \langle \mathbf{I},s_2 \rangle + \langle \delta \mathbf{I}, b \rangle=\langle \mathbf{I},s_2 \rangle$, since, due to (\ref{continuity}), $\delta \mathbf{I}=\mathbf{0}$.
\end{proof}

Due to Theorem~\ref{thlc}, one can state the following definition:
\begin{definition}\label{lcNew}
The value $I_c = \langle \mathbf{I},s \rangle$, for an arbitrary $1$-cycle $c \in Z_1(\mathcal{K}_a)$ and a $2$-chain $s$ such that $c=\partial s$, is called current \textit{linked} by the $1$-cycle $c$.
\end{definition}

Let us now show that the current linked by a $1$-cycle $c \in Z_1(\mathcal{K}_a)$ is the same for all the cycles in the homology class of $c$.
\begin{theorem}\label{lc2}
The linked current $I_c = \langle \mathbf{I},s \rangle$ depends only on the $H_1(\mathcal{K}_a)$ class of the $1$-cycle $c=\partial s$, where $c \in Z_1(\mathcal{K}_a)$. Thanks to non-local Amp\`{e}re's law, the same holds for the dot product $\langle \mathbf{F},c \rangle$.
\end{theorem}
\begin{proof}
Let us take two cycles $c_1, c_2 \in \mathcal{K}_a$ in the same homology class $H_1(\mathcal{K}_a)$. It means that $c_1=c_2+\partial s$ holds, where $s \in C_2(\mathcal{K}_a)$.
Then, $\langle \mathbf{F},c_1 \rangle=\langle \mathbf{F},c_2 \rangle + \langle \mathbf{F},\partial s \rangle = \langle \mathbf{F},c_2 \rangle + \langle \delta \mathbf{F}, s \rangle=\langle \mathbf{F},c_2 \rangle + \langle \mathbf{I}, s \rangle$. From the definition of $\mathbf{I}$ we have that $\langle \mathbf{I}, f \rangle = 0$ for every $f \in \mathcal{K}_a$, therefore $\langle \mathbf{I}, s \rangle = 0$, what completes the proof.
\end{proof}

The idea of the proof of Theorem~\ref{lc2} is presented in Fig.~\ref{fig:h1classes} for the complement of a solid double torus with respect to a cube which contains it (not represented in the picture for the sake of clarity). The supports of two cycles $c_1$ and $c_2$ in the same homology class are depicted in Fig. \ref{fig:h1classes}a. The supports of two $2$-chains $s_1$ and $s_2$ used to evaluate the linked currents $I_{c_1} = \langle \mathbf{I},s_1 \rangle$ and $I_{c_2} = \langle \mathbf{I},s_2 \rangle$ are depicted in Fig. \ref{fig:h1classes}b. Finally, in Fig. \ref{fig:h1classes}c it is possible to see the support of a $2$-chain $s$ such that $c_1=c_2+\partial s$, which has been used in the proof.
\begin{figure} [!h]
\centering
\includegraphics[width=13cm]{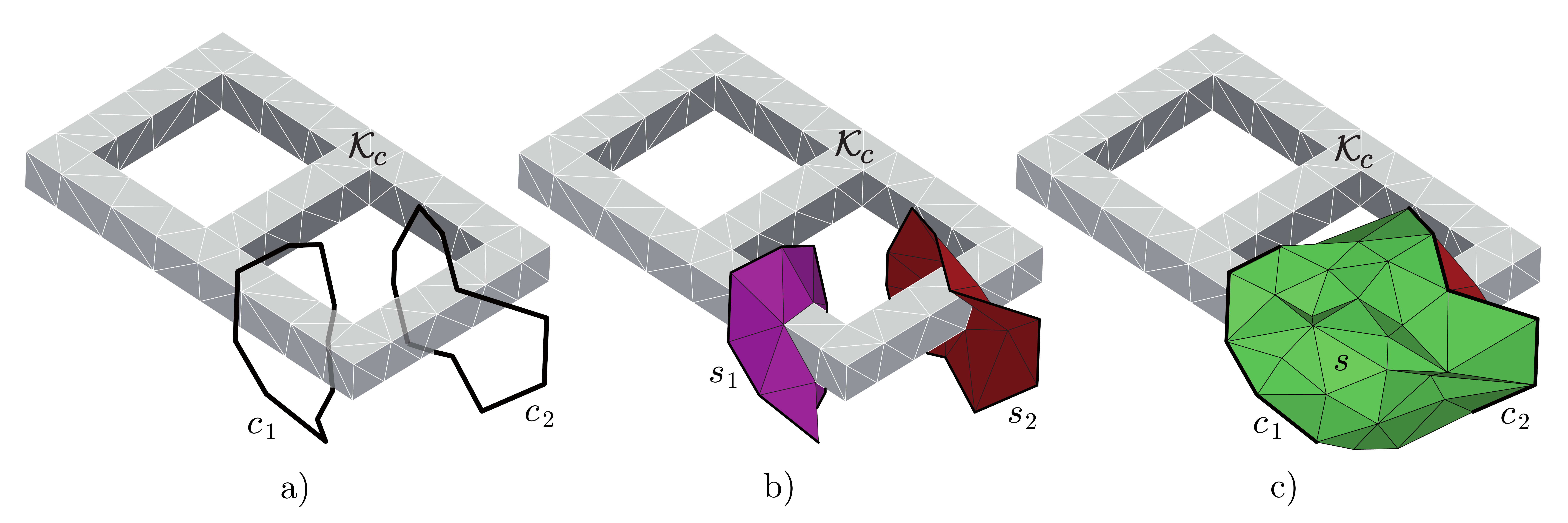}
\caption{a) Two cycles, $c_1$ and $c_2$, in the same $H_1(\mathcal{K}_a)$ class. b) Example of two possible $2$-chains used to evaluate the currents linked to $c_1$ and $c_2$. c) The $2$-chain $s$ whose boundary $1$-chain is $c_1-c_2$.} \label{fig:h1classes}
\end{figure}

Using the scalar potential in $\mathcal{K}_a$, as the definition (\ref{potential2}), yields to an inconsistency, since Amp\`{e}re's law may be violated on some $1$-cycle $c=\partial s \in Z_1(\mathcal{K}_a)$.
In fact, let us consider a $1$-cycle $c \in H_1(\mathcal{K}_a)$.
Using the double torus example previously introduced, this cycle may be for example the cycle $c_1$ or $c_2$ in Fig. \ref{fig:h1classes}a.
Now, due to (\ref{potential2}), we have
$$\langle \mathbf{F},c \rangle = \langle \mathbf{F} , \partial s \rangle = \langle \mathbf{T} + \delta \boldsymbol{\Omega},\partial s \rangle = \langle \mathbf{T}, \partial s \rangle + \langle \delta \boldsymbol{\Omega},\partial s \rangle=\langle \delta \delta \boldsymbol{\Omega},s \rangle = \langle \mathbf{0} , s \rangle =
0 \neq \langle \mathbf{I},s \rangle,$$
since $\langle \mathbf{T}, e \rangle = 0$ for every $e\in \mathcal{K}_a$.
The last inequality $0 \neq \langle \mathbf{I},s \rangle$ follows from the fact that the current flowing through $s$ is non-zero in general, since the support of $s$ has to intersect $\mathcal{K}_c$. According to next Theorem, the presented problem may occur only with cycles which are non-trivial in $H_1(\mathcal{K}_a)$.
\begin{theorem}
Amp\`{e}re's law is satisfied for every $1$-boundary $b$ in $\mathcal{K}_a$. The cycles that produce an inconsistency in Amp\`{e}re's law---because may link a non-zero current---are the cycles which are non-trivial in the $1$-st homology group $H_1(\mathcal{K}_a)$.
\end{theorem}
\begin{proof}
In the first case, the cycle $c \in B_1(\mathcal{K}_a)$ is \textit{bounding}, which means that there exists $s \in C_2(\mathcal{K}_a)$ such that $\partial s = c$.
Since $s$ does not intersect $\mathcal{K}_c$, the dot product of the current $\mathbf{I}$ and $s$ is zero and the dot product of the m.m.f. $\mathbf{F}$ and $c$ is also zero.

In the second case, $c$ being non-zero in the first homology group $H_1(\mathcal{K}_a)$, such a chain $s \in C_2(\mathcal{K}_a)$ does not exist. But, since the whole complex $\mathcal{K}$ is homologically trivial, there exist a $2$-chain $s' \in C_2(\mathcal{K})$ such that $\partial s'=c$. Consequently one needs to have $|s'| \cap \mathcal{K}_c \not = \emptyset$. Thus, the support of $s'$ has to extend in the current-carrying region $\mathcal{K}_c$ and, consequently, the current trough $s'$ (and linked by $c$) is non-zero in general. This causes in general an inconsistency in Amp\`{e}re's law for the cycles $c$ non-trivial in $H_1(\mathcal{K}_a)$.
\end{proof}

\begin{figure} [!h]
\centering
\includegraphics[width=11cm]{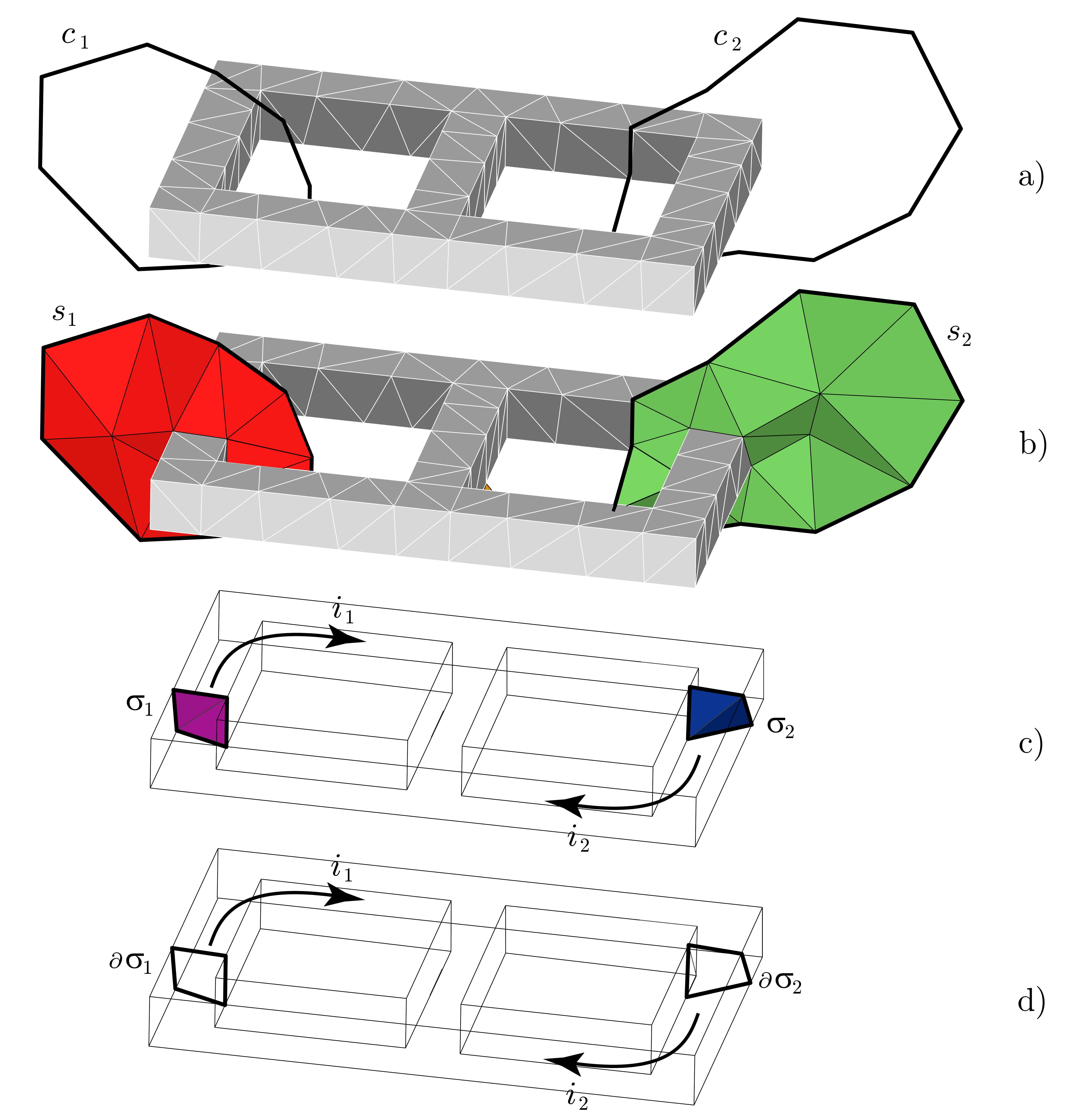}
\caption{a) Two non-trivial elements of $H_1(\mathcal{K}_a)$ called $c_1$ and $c_2$. b) Let us consider $2$-chains in $\mathcal{K}$ whose boundaries are $c_1$ and $c_2$. c) Two independent currents, $i_1$ and $i_2$, evaluated through $2$-chains $s_1$ and $s_2$ or through $2$-chains $\sigma_1$ and $\sigma_2$. d) The $1$-cycles $\partial \sigma_1$ and $\partial \sigma_2$ are in the same homology $H_1(\mathcal{K}_a)$ class as the corresponding $1$-cycle $c_1$ and $c_2$. This is formalized in the following of the paper.
} \label{fig:currents}
\end{figure}

The idea of last proof can be presented by using again the double torus example. Let us introduce two non-trivial elements of $H_1(\mathcal{K}_a)$ called $c_1$ and $c_2$ which are the representatives of generators of the first homology group $H_1(\mathcal{K}_a)$ and whose support is depicted in Fig. \ref{fig:currents}a.
The support of two $2$-chains $s_i$, $i \in \{1,2\}$, such that $\partial s_i=c_i$, are represented in Fig. \ref{fig:currents}b. It is easy to see that the supports of these $2$-chains have to intersect $\mathcal{K}_c$.

\subsection{Independent currents}\label{sec:ic}
After showing the inconsistency arising in general from the definition of potentials presented in Section \ref{spre}, some modifications in this definition are needed for Amp\`{e}re's law to hold implicitly for every $1$-cycle $c \in Z_1(\mathcal{K}_a)$ and a $2$-chain $s$ such that $c=\partial s$.
We define in this Section a key ingredient for the new definition of potentials, which is called set of \textit{independent currents}.

It has been already pointed out that the electric current is nonzero inside the conducting region $\mathcal{K}_c$ only. Consequently, for every chain $s \in C_2(\mathcal{K})$ for which $\partial s=c \in Z_1(\mathcal{K}_a)$, the current can be non-zero only on $|s| \cap \mathcal{K}_c$. Since $\mathcal{K}_c$ is a sub-complex of $\mathcal{K}$, the restriction of the $2$-chain $s$ to $\mathcal{K}_c$, denoted as $\sigma=s|_{\mathcal{K}_c}$, is an element of $C_2(\mathcal{K}_c)$. The supports of such restrictions of the $2$-chains $s_1$ and $s_2$ in the double torus example are the $\sigma_1$ and $\sigma_2$ shown in Fig. \ref{fig:currents}c.
Since $\partial \sigma \in C_1(\partial \mathcal{K}_c)$, we have that $\sigma \in Z_2(\mathcal{K}_c,\partial \mathcal{K}_c)$.
Consequently $\sigma$ can be generated from $H_2(\mathcal{K}_c,\partial \mathcal{K}_c)$ basis by adding the boundary of a suitable $3$-chain $d \in C_3(\mathcal{K}_c)$. The current through any chain non-zero in $H_2(\mathcal{K}_c,\partial \mathcal{K}_c)$ is determined by the current through the $H_2(\mathcal{K}_c,\partial \mathcal{K}_c)$ basis elements (for trivial $2$-chains and relative $2$-chains the current is zero as a consequence of local version of current continuity law~(\ref{continuity})).

For the whole Section, let us fix the set of relative cycles $\sigma_1,\ldots,\sigma_n \in Z_2(\mathcal{K}_c, \partial \mathcal{K}_c)$ representing the homology group $H_2(\mathcal{K}_c, \partial \mathcal{K}_c)$ generators.
From the Theorem~\ref{th:relHomologyExactSequence} one may write the following exact sequence of the pair $(\mathcal{K},\mathcal{K}_a)$:
$$\ldots \xrightarrow{\partial} H_2(\mathcal{K}_a) \xrightarrow{i_{*}} H_2(\mathcal{K}) \xrightarrow{j_{*}} H_2(\mathcal{K},\mathcal{K}_a) \xrightarrow{\partial} H_{1}(\mathcal{K}_a) \xrightarrow{i_{*}} \ldots.$$
The assumption that the mesh $\mathcal{K}$ is acyclic provides $H_2(\mathcal{K}) = 0$. Consequently, $\partial : H_2(\mathcal{K},\mathcal{K}_a) \rightarrow H_{1}(\mathcal{K}_a)$ is an isomorphism defined in the following way $\partial : H_2(\mathcal{K},\mathcal{K}_a) \ni [\alpha] \rightarrow [\partial \alpha] \in H_{1}(\mathcal{K}_a)$.

Let us use Theorem~\ref{th:HatcherCorollary224} for the sub-complex $A$ equal to $\mathcal{K}_a$ and sub-complex $B$ equal to $\mathcal{K}_c$ and $X$ equal to $\mathcal{K}$. Since $\mathcal{K}_a \cap \mathcal{K}_c = \partial \mathcal{K}_c$, the following inclusion map $(\mathcal{K}_c,\partial \mathcal{K}_c)\hookrightarrow(\mathcal{K},\mathcal{K}_a)$ induces the
isomorphism $H_2(\mathcal{K}_c,\mathcal{K}_c\cap \mathcal{K}_a)  =  H_2(\mathcal{K}_c,\partial \mathcal{K}_c) \rightarrow H_2(\mathcal{K},\mathcal{K}_a)$. Consequently, we have the following sequence of isomorphisms
$$H_2(\mathcal{K}_c,\partial \mathcal{K}_c) \xrightarrow{(\mathcal{K}_c,\partial \mathcal{K}_c)\hookrightarrow(\mathcal{K},\mathcal{K}_a)} H_2(\mathcal{K},\mathcal{K}_a) \xrightarrow{\partial} H_{1}(\mathcal{K}_a),$$
where on the left-hand side there is the group generated by the $H_2(\mathcal{K}_c,\partial \mathcal{K}_c)$ generators and on the right-hand side there is the group generated by the classes of cycles on which the Amp\`{e}re's law has to be enforced. The image of the generators of $H_2(\mathcal{K}_c,\partial \mathcal{K}_c)$ through the above isomorphism is referred to as \textit{independent cycles}.

What we showed in this Section motivates the following definition:
\begin{definition}
The complex numbers being the dot-products of the current $\mathbf{I}$ with the representatives of a basis $\{\sigma_j\}$ of $H_2(\mathcal{K}_c,\partial \mathcal{K}_c)$ are called \textit{independent currents} $\{i_j\}$ in $\mathcal{K}_c$
$$i_j=\langle \mathbf{I}, \sigma_j \rangle, j \in \{1,\ldots,\beta_1(\mathcal{K}_a)\}.$$
\end{definition}
We want to point out that the $\{\sigma_j\}$ are integer $H_2(\mathcal{K}_c,\partial \mathcal{K}_c)$ homology generators. When evaluating the dot-product, we trivially interpret them as complex homology group generators.

Since the isomorphism $H_2(\mathcal{K}_c,\partial \mathcal{K}_c) \rightarrow H_2(\mathcal{K},\mathcal{K}_a)$ is induced by the inclusion map, $\sigma_1,\ldots,\sigma_n$ form also a set of cycles representing the basis of $H_2(\mathcal{K},\mathcal{K}_a)$. When passing by the isomorphism induced by the boundary map $\partial : H_2(\mathcal{K},\mathcal{K}_a) \rightarrow H_1(\mathcal{K}_a)$, one gets that $\partial \sigma_1,\ldots,\partial \sigma_n$ is a set of cycles representing a basis in $H_{1}(\mathcal{K}_a)$.

\subsection{Definition of cuts}\label{sec:defOfCuts}
First of all, let us note that it suffices to enforce Amp\`{e}re's law on the cycles $[\partial \sigma_1],\ldots,[\partial \sigma_n]$. Then, for every other cycle $[c] \in H_{1}(\mathcal{K}_a)$ such that $[c] = [\sum_{i=1}^n \lambda_i \partial \sigma_i]$, the current linked by $c$ is equal to $\sum_{i=1}^n \lambda_i \langle \mathbf{I} , \sigma_i \rangle=\sum_{i=1}^n \lambda_i i_i$, which follows from the fact that dot product of the m.m.f. on boundaries is zero.

Now, taking into account the arguments in the last Section, we would like to modify the definition (\ref{potential2}) of $\mathbf{F}$ in $\mathcal{K}_a$ in such a way that Amp\`{e}re's law is satisfied for all cycles $c \in Z_1(\mathcal{K}_a)$. Since $\mathbf{F}$ is a $1$-cocycle in $\mathcal{K}_a$, we are going to construct a family of $1$-cocycles $\{\mathbf{c}^i\}_{i=1}^n$ in $\mathcal{K}_a$ over $\mathbb{Z}$ which, after being multiplied by the independent currents $\{i_j\}_{i=1}^n$, are added to $\mathbf{T}$.
In particular, the family of $1$-cocycles $\{\mathbf{c}^i\}_{i=1}^n$ should verify $\langle \mathbf{c}^i,\partial \sigma_j \rangle = \delta_{ij}$ for every $i,j \in \{1,\ldots,n\}$.
%
%
We are going to show that, for this purpose, the representatives of a basis of the $1$-st cohomology group $H^1(\mathcal{K}_a)$ dual to the $H_1(\mathcal{K}_a)$ basis $[\partial \sigma_1],\ldots,[\partial \sigma_n]$ are needed. To prove the existence of the dual basis, let us recall the Universal Coefficient Theorem for cohomology.

\begin{theorem}[\cite{Hatcher}, Theorem 3.2]
\label{th:unicoefcohomo}
If a complex $\mathcal{K}_a$ has (integer) homology groups $H_n(\mathcal{K}_a)$, then the cohomology groups $H^n(\mathcal{K}_a , G)$ are determined by splitting exact sequences
$$
0 \rightarrow Ext(H_{n-1}(\mathcal{K}_a), G) \rightarrow H^{n}(\mathcal{K}_a,G)
\xrightarrow{h} Hom( H_n(\mathcal{K}_a) , G) \rightarrow 0.
$$
\end{theorem}

In this paper, there is no need to go into the definition of the
$Ext$ functor. The key property is that $Ext(Q,G) = 0$
if $Q$ is a free group. For further details and proof of this
property consult \cite{Hatcher}.

For a class
$[d] \in H^n(\mathcal{K}_a,G)$, since $d$ is a cocycle,
one has $0 = \langle \delta d, z \rangle = \langle d,\partial z \rangle$ for every $z \in Z_{n+1}(\mathcal{K}_a)$. From the above
equality, it follows that $d|_{B_n(\mathcal{K}_a)} = 0$. Let us
define the restriction $d_0 = d|_{Z_n(\mathcal{K}_a)}$. Since
$d_0|_{B_n(\mathcal{K}_a)} = 0$, then $d_0 \in Hom( H_n(\mathcal{K}_a) , G )$.
This shows the correctness of the definition of the map
$h([d]) = d_0 \in Hom( H_n(\mathcal{K}_a) , G )$ used in the exact sequence in Theorem~\ref{th:unicoefcohomo}.
Finally, we need to show that, for $d,d' \in [d]$, one has $h(d) = h(d')$. Since $d,d' \in [d]$, there exists $e \in C^0(K)$ such that $d = d' + \delta e$. Let us take a cycle $f \in C_1(\mathcal{K})$, then we have $\partial f = 0$. Consequently $\langle d , f \rangle = \langle d'+\delta e , f \rangle = \langle d' , f \rangle + \langle \delta e , f \rangle = \langle d' , f \rangle$. Therefore, the value of the map $h$ does not depend on the representatives of the cocylce and $h$ is well defined.

In our case the group $G$ is the group of integers and the Universal
Coefficient Theorem for cohomology is used for $n=1$.
In this case, the exact sequence from Theorem~\ref{th:unicoefcohomo} has the form
\[
0 \rightarrow Ext(H_{0}(\mathcal{K}_a), \mathbb{Z}) \rightarrow
H^{1}(\mathcal{K}_a, \mathbb{Z}) \xrightarrow{h} Hom( H_1(\mathcal{K}_a) ,
\mathbb{Z}) \rightarrow 0.
\]

For the complex $\mathcal{K}_a$ we have that  $H_0(\mathcal{K}_a) = \mathbb{Z}^p$ for some $p \in \mathbb{Z}$, $p >0$. This provides that
$H_0(\mathcal{K}_a)$ is a free group. From the cited property of the $Ext$ functor, it follows that
$Ext(H_{0}(\mathcal{K}_a), \mathbb{Z}) = 0$. From the exactness of the sequence, one has that $h : H^{1}(\mathcal{K}_a,\mathbb{Z}) \rightarrow Hom(H_1(\mathcal{K}_a) , \mathbb{Z})$ is an isomorphism.

Due to Theorem~\ref{th:homologyAreTorsionFree} the homology group $H_1(\mathcal{K}_a)$ is torsion free. This provides, from Theorem 3.61 in~\cite{CompHom}, that it is isomorphic to the direct sum of $\dim(H_1(\mathcal{K}_a)) = n$ copies of $\mathbb{Z}$. From the set of cycles $\partial \sigma_1,\ldots,\partial \sigma_n$ 
forming a $H_1(\mathcal{K}_a)$ basis, a set of functions $\zeta_i$, $i \in \{1,\ldots,n\}$ such that $\zeta_i([\partial \sigma_j]) = \delta_{ij}$ form a basis of $Hom(H_1(\mathcal{K}_a) , \mathbb{Z})$. From the description of the isomorphism $h : H^{1}(\mathcal{K}_a,\mathbb{Z}) \rightarrow Hom(H_1(\mathcal{K}_a) , \mathbb{Z})$, it is straightforward that $h^{-1}(\zeta_i)$ is a cochain $\mathbf{c}^i$ being an element of the $H^{1}(\mathcal{K}_a,\mathbb{Z})$ basis we are looking for.

In the presented reasoning we have started from the set of independent currents and end up to the cohomology basis. The Reader should be aware that exactly the same reasoning can be made the other way around. Namely, if one starts from the $H^1(\mathcal{K}_a)$ basis, then is able to find the corresponding $H_2(\mathcal{K},\mathcal{K}_a)$ basis, which directly corresponds to some basis $H_2(\mathcal{K}_c,\partial \mathcal{K}_c)$ yielding a set of independent currents.

The following Theorem follows easily form what has been already said.
\begin{theorem}
\label{th:defOfCuts}
Let $\{ \mathbf{c}^i \}_{i=1}^{\beta_1(\mathcal{K}_a)}$ be the cocycles representing the $H^1(\mathcal{K}_a)$ basis.
Let $\{\partial \sigma_i\}_{i=1}^{\beta_1(\mathcal{K}_a)}$ be the cycles representing the dual $H_1(\mathcal{K}_a)$ basis.
Once we redefine the m.m.f. as $\mathbf{F} = \delta \boldsymbol{\Omega}+\mathbf{T} + \sum_{i=1}^{\beta_1(\mathcal{K}_a)} i_i \mathbf{c}^i$, then the current linked by the cycles in the homology class of $\partial \sigma_i$ is equal to $i_i$. Amp\`{e}re's law (\ref{nl_ampere}) holds for every $1$-cycle $c \in Z_1(\mathcal{K}_a)$. Hence, the potentials are now consistently designed.
\end{theorem}

Further in this paper we assume that the cocycles $\{ \mathbf{c}^i \}_{i=1}^{\beta_1(\mathcal{K}_a)}$ representing the $H^1(\mathcal{K}_a)$ basis are defined in the whole complex $\mathcal{K}$. Therefore, we assume that $\langle \mathbf{c}^i , e \rangle = 0$ for every $i \in \{1,\ldots,\beta_1(\mathcal{K}_a)\}$ and for every $e \in \mathcal{K}_c \setminus \mathcal{K}_a$.

Thus we are now able to give the definition of cuts:
\begin{definition}
The cuts $\{ \mathbf{c}^j \}_{j=1}^{\beta_1(\mathcal{K}_a)}$ are defined as representatives of the first cohomology group generators over integers of the insulating region $\mathcal{K}_a$. \end{definition}

\begin{figure} [!h]
\centering
\includegraphics[width=13cm]{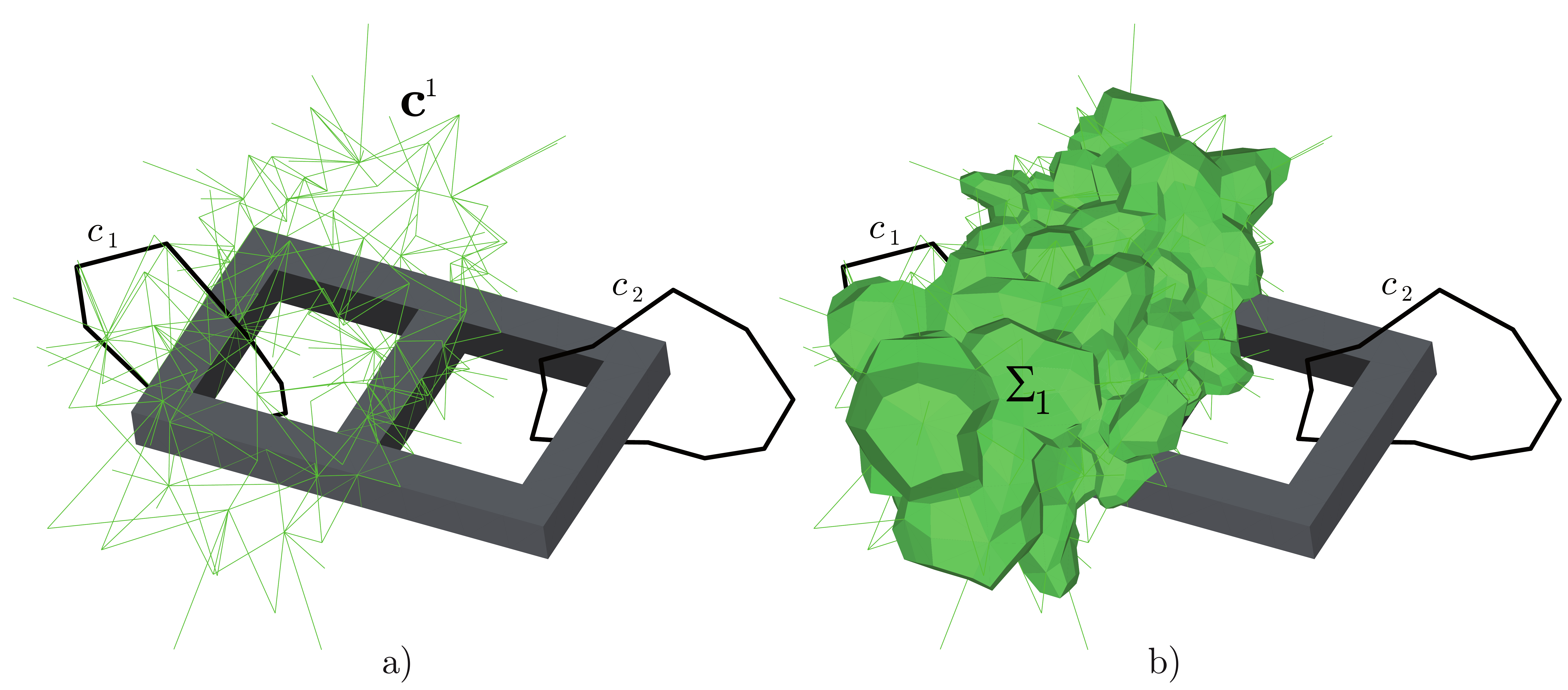}
\caption{a) The two generators $c_1$ and $c_2$ and the support of the $\mathbf{c}^1$ $1$-cochain is shown for the double torus example. b) The dual faces $f_{\mathcal{B}}$, dual to edges in the support of $\mathbf{c}^1$, from a $2$-chain $\Sigma_1$ on the barycentric complex.} \label{fig:generators}
\end{figure}
To illustrate the presented idea let us consider the two fixed generators $c_1$ and $c_2$ for $H_1(\mathcal{K}_a)$ relative to the previous example, see Fig. \ref{fig:generators}a. In the same Figure, a concrete example of a representative of a $H^1(\mathcal{K}_a)$ cohomology generator---dual to the $H_1(\mathcal{K}_a)$ generator $c_1$---is shown. The edges in the picture are the ones that constitute the support of $\mathbf{c}^1$. The integers associated to these edges are given in such a way that $\langle \mathbf{c}^1 , c_1 \rangle = 1$ and $\langle \mathbf{c}^1 , c_2 \rangle = 0$.

\subsubsection{Interpretation of cuts on the dual complex}
\label{sec:cutsOnDualCmplx}
Some results presented in this Section are part of an old-fashioned proof of Poincar\'{e}--Lefschetz duality for manifolds with boundary. For this paper, the following restricted version of the duality is used:
\begin{theorem}[Poincar\'{e}--Lefschetz duality]
$H^1(\mathcal{K}_a) \cong H_2(\mathcal{B}_a,\partial \mathcal{B}_a)$.
\end{theorem}

The modern proof of this famous Theorem bases on the idea of cup product, as for instance in~\cite{Hatcher}.
However, the original proof proposed by Poincar\'{e} himself\footnote{Which turned out not to be
complete, but was corrected later on by the Poincar\'{e}'s successors.}, is based on the concept of dual cell structure, which has been described in this paper in Section~\ref{sec:ATinCE}. For the classical proof of Poincar\'{e}--Lefschetz duality, one may consult
\cite{berdon} or \cite{seifThrel}. In this proof the \textit{dualization operator} $D$ is defined on the complex
$\mathcal{K}_a$ in the way that, for a $1$-cell $c \in \mathcal{K}_a$, the corresponding dual $2$-cell $Dc \in \mathcal{B}_a$ is
assigned. The presented map turns out to induce the isomorphism in the Poincar\'{e}--Lefschetz duality (for further details consult~\cite{berdon}).

Consequently, form the Poincar\'{e}--Lefschetz duality, once the $1$-cocycles that represent a basis $\{\mathbf{c}^1,\ldots,\mathbf{c}^n\}$ of $H^1(\mathcal{K}_a)$ are provided, it is clear that the set of dual $2$-cycles $d_1,\ldots,d_n$ defined in the following way:
$$ d_i = \sum_{S \in \mathcal{B}_a} \alpha^s_i S, \text{ where } \alpha^s_i = \langle \mathbf{c}^i,D^{-1}S \rangle$$
are the relative cycles that represent a basis of $H_2(\mathcal{B}_a,\partial \mathcal{B}_a)$. These cycles are denoted as $\{\Sigma_i\}_{i=1}^{\beta_1(\mathcal{K}_a)}$. The visualization of the presented duality for the proposed example can be seen in Fig.~\ref{fig:generators}. On the left, the cohomology generator $\mathbf{c}^i$ is represented, while, on the right, the representative $\Sigma_i$ of the generator of $H_2(\mathcal{B}_a,\partial \mathcal{B}_a)$ is depicted.

\section{$T$-$\Omega$ magneto-quasi-static formulation}
\label{sec:F-Omega}
After the potential design, in this Section we analyze how to solve the magneto-quasi-static BVP.

\subsection{Constitutive matrices}
\label{sec:constitutiveMatrices}
The discrete counterparts of the constitutive laws links $k$-cochains in $\mathcal{K}$ with $(3-k)$-cochains in $\mathcal{B}$:
\begin{equation}\label{tomega2}
\begin{array}{llll}
   \tilde{\boldsymbol{\Phi}}=\boldsymbol{\mu} \, \mathbf{F}\,\,\,\, \: \text{  (a),}& \,\,\,\,\,\,\,\,
   \left.\tilde{\mathbf{U}}\right|_{\mathcal{B}_c}=\boldsymbol{\rho}\,\left.\mathbf{I}\right|_{\mathcal{K}_c}\,\,\,\, \: \text{  (b).}
\end{array}
\end{equation}
The constitutive matrices provide a relation between cochains on the primal and cochains on the dual complex (see Section~\ref{sec:ATinCE}).
The square matrix $\boldsymbol{\mu}$ is called \textit{permeance matrix} and is the approximate discrete counterpart of the constitutive
relation $\mathrm{B}=\mu\,\mathrm{H}$ at continuous level, $\mu$ being the
permeability assumed element-wise constant and $\mathrm{H}$ and $\mathrm{B}$ are the magnetic field and the magnetic flux density vector fields, respectively.
The square matrix $\boldsymbol{\rho}$ is called \textit{resistivity matrix} and is the approximate discrete
counterpart of the constitutive relation $\mathrm{E}=\rho\,\mathrm{J}$ at
continuous level, $\rho$ being the resistivity assumed element-wise a constant and $\mathrm{E}$ and $\mathrm{J}$ are the electric field and the current density vector fields, respectively. $\boldsymbol{\rho}$ is defined to be zero for geometric elements in $\mathcal{K}_a \setminus \mathcal{K}_c$.

Describe in detail how to construct the constitutive matrices $\boldsymbol{\rho}$ and $\boldsymbol{\mu}$ goes beyond the
purpose of this paper. Methods valid for a general polyhedral mesh are described for
example in \cite{cmamepoly}, \cite{jcppoly} and references therein.

\subsection{Algebraic equations}
\label{sec:AlgebraicEquations}
In this Section, the constitutive matrices described in the Section~\ref{sec:constitutiveMatrices} are combined with the local algebraic laws presented in Section~\ref{sec:MaxwellAlgForm} to obtain an algebraic system of equations.

Up to now, we know that Amp\`{e}re's law holds in $\mathcal{K}_a$ and current continuity law holds in $\mathcal{K}$. Hence, we have to enforce the other laws by means of a linear system of equations.

To do this, let us start from the magnetic Gauss's law~(\ref{gauss}) $\langle \tilde{\delta} \tilde{\boldsymbol{\Phi}}, v_{\mathcal{B}} \rangle=0$ and let us use the constitutive relation~(\ref{tomega2}a) $\tilde{\boldsymbol{\Phi}}=\boldsymbol{\mu} \, \mathbf{F}$. Consequently, we get $\langle \tilde{\delta} \boldsymbol{\mu} \, \mathbf{F}, v_{\mathcal{B}} \rangle=0$. The definition of potentials $\mathbf{F} = \delta \boldsymbol{\Omega}+\mathbf{T} + \sum_{j=1}^{\beta_1(\mathcal{K}_a)} i_j \mathbf{c}^j$ from Theorem~\ref{th:defOfCuts} is substituted in the last equation. In this way, the final equation is obtained:
\begin{equation}\label{eq1prime}
\tilde{\delta} \boldsymbol{\mu} \delta \boldsymbol{\Omega}
+
\tilde{\delta} \boldsymbol{\mu} \mathbf{T}
+
 \sum_{j=1}^{\beta_1(\mathcal{K}_a)} \tilde{\delta} \boldsymbol{\mu} \mathbf{c}^j\,i_j  = \mathbf{0}.
\end{equation}

Now, Faraday's law has still to be enforced in the conducting region (due to the boundary condition $\langle \mathbf{T}, e\rangle=0$ on edges $e \in \partial \mathcal{K}_c$, the e.m.f. is not needed in the insulating region by using this formulation, see for example \cite[p. 1030]{carpenter} for a more detailed explanation).
To do this, let us start from the local Faraday's law~(\ref{faraday})
$\langle \tilde{\delta} \tilde{\mathbf{U}} + i \omega \tilde{\boldsymbol{\Phi}}, f_{\mathcal{B}} \rangle = 0$, $\forall f_{\mathcal{B}} \in \mathcal{B}_c$.
Now, let us substitute the constitutive relations~(\ref{tomega2}) 
in the above equation. In this way, we obtain $\langle \tilde{\delta} \boldsymbol{\rho} \mathbf{I} + i \omega \boldsymbol{\mu} \mathbf{F}, f_{\mathcal{B}} \rangle = 0$, $\forall f_{\mathcal{B}} \in \mathcal{B}_c$.
Let us now use the local Amp\`{e}re's law (\ref{ampere}) in $\mathcal{K}_c$ by substituting $\delta \mathbf{F}= \mathbf{I}$
\[\langle \tilde{\delta} \boldsymbol{\rho} \delta \mathbf{F} + i \omega \boldsymbol{\mu} \mathbf{F}, f_{\mathcal{B}} \rangle = 0, \forall f_{\mathcal{B}} \in \mathcal{B}_c.\]
For the sake of brevity, let us define a matrix $\mathbf{R}=\tilde{\delta} \boldsymbol{\rho} \delta + i \omega \boldsymbol{\mu}$. Hence, we can write
$$\langle \mathbf{RF}, f_{\mathcal{B}}\rangle=\langle \mathbf{R} \left( \delta \boldsymbol{\Omega}+\mathbf{T} + \sum_{j=1}^{\beta_1(\mathcal{K}_a)} i_j \mathbf{c}^j \right), f_{\mathcal{B}}\rangle=0, \forall f_{\mathcal{B}} \in \mathcal{B}_c.$$
Since $\mathbf{R} \delta \boldsymbol{\Omega}=i \omega \boldsymbol{\mu} \delta \boldsymbol{\Omega}$, we obtain:
\begin{equation}\label{eq2prime}
\langle i \omega \boldsymbol{\mu} \delta \boldsymbol{\Omega} + \mathbf{R T} + \sum_{j=1}^{\beta_1(\mathcal{K}_a)} \mathbf{R} \mathbf{c}^j i_j, f_{\mathcal{B}} \rangle = 0, \forall f_{\mathcal{B}} \in \mathcal{B}_c.
\end{equation}
The unknowns are the coefficients of the cochain $\boldsymbol{\Omega}$ associated with each node and the coefficients of $\mathbf{T}$ associated with edges belonging to $\mathcal{K}_c \setminus \mathcal{K}_a$ (in fact, $\langle \mathbf{T},e\rangle=0$, $\forall e \in \mathcal{K}_a$). We have written the equations (\ref{eq1prime}) and (\ref{eq2prime}) corresponding to these unknowns.
But we have also the independent currents $i_j$ as additional unknowns. Which are the needed additional equations and where do they come from?

\subsection{Non-local Faraday's equations and the final linear system of equations}
The dot product of the e.m.f. $\langle \tilde{\mathbf{U}},b \rangle$ with every bounding $1$-cycle $b \in C_1(\mathcal{B}_c)$ is easily determined by using a non-local Faraday's law
$$\langle\tilde{\mathbf{U}},b \rangle=\langle -i \omega \tilde{\boldsymbol{\Phi}},s \rangle, \,\,b \in B_1(\mathcal{B}_c) \,\, \mathrm{and}\,\, b=\partial s,$$
which is a linear combination of local Faraday's laws already enforced by (\ref{eq2prime}).
Similarly to what developed about the independent currents in section \ref{sec:ic}, the linked flux $\tilde{\Phi}_c=\langle \tilde{\boldsymbol{\Phi}},s \rangle$, linked by the cycle $b$, does not depend on the $2$-chain $s$. This is because the local Gauss's magnetic law (\ref{gauss}) hold thanks to (\ref{eq1prime}). Hence, such non-local equations written on boundaries do not bring any new constraint.

On the opposite, $\langle \tilde{\mathbf{U}},h \rangle$ over a $1$-cycle $h \in C_1(\mathcal{B}_c)$ nonzero in $H_1(\mathcal{B}_c)$ cannot be determined by using only cochains in $\mathcal{B}_c$. This is because $\langle \tilde{\mathbf{U}},h \rangle$ depends on cochains in $\mathcal{B}_c$ and $\mathcal{B}_a$ through the non-local Faraday's law. Namely $\langle \tilde{\mathbf{U}},h\rangle$ has to match the magnetic flux variation $\langle - i \omega \tilde{\boldsymbol{\Phi}},s \rangle$ through a $2$-chain $s$ such that $c=\partial s$. The key point is that the support of $s$ extends also in the sub-complex $\mathcal{B}_a$.

We need to show that the $2$-chains used to take dot product with the magnetic flux, whose boundary are the $H_1(\mathcal{B}_c)$ generators, are generators for $H_2(\mathcal{B}, \mathcal{B}_c)$. Similarly to what done with the currents, we need just the $H_2(\mathcal{B}_a, \partial \mathcal{B}_a)$ generators and the reasoning presented in this Section is analogous to the one presented in Section~\ref{sec:defOfCuts}.
Let us take a fixed set of $1$-cocycles $\{\mathbf{c}^1,\ldots,\mathbf{c}^n\}$, $c^i \in C^1(\mathcal{K}_a)$ for $i \in \{1,\ldots,n\}$, representing the basis of $H^1(\mathcal{K}_a)$. The set of cycles $\{d_1,\ldots,d_n\}$, $d_i \in C_2(\mathcal{B}_a,\partial \mathcal{B}_a)$ for $i \in \{1,\ldots,n\}$, defined in the Section~\ref{sec:cutsOnDualCmplx}, represents a basis of $H_2(\mathcal{B}_a,\partial \mathcal{B}_a)$.
Let us write the long exact sequence of the pair $(\mathcal{B}, \mathcal{B}_c)$:
$$\ldots \xrightarrow{\partial} H_n(\mathcal{B}_c) \xrightarrow{i_{*}} H_n(\mathcal{B}) \xrightarrow{j_{*}} H_n(\mathcal{B} , \mathcal{B}_c) \xrightarrow{\partial} H_{n-1}(\mathcal{B}_c) \xrightarrow{i_{*}} \ldots.$$
Since $\mathcal{B}$ is acyclic, $H_n(\mathcal{B})$ is trivial. Therefore, from the exactness of the sequence, $\partial : H_n(\mathcal{B} , \mathcal{B}_c) \rightarrow H_{n-1}(\mathcal{B}_c)$ is an isomorphism.
Let us now use the Theorem~\ref{th:HatcherCorollary224} for $X = \mathcal{B}$, $A = \mathcal{B}_c$ and $B = \mathcal{B}_a$ and $n=2$. This gives us the isomorphism $H_2(\mathcal{B}_a,\mathcal{B}_a \cap \mathcal{B}_c) = H_2(\mathcal{B}_a,\partial \mathcal{B}_a) \hookrightarrow H_1(\mathcal{B} , \mathcal{B}_c)$.
Consequently, we have the sequence of isomorphisms $H_2(\mathcal{B}_a,\partial \mathcal{B}_a) \hookrightarrow H_2(\mathcal{B} , \mathcal{B}_c) \xrightarrow{\partial} H_{1}(\mathcal{B}_c)$.
Therefore, the set of cycles $\{\partial d_1,\ldots,\partial d_n\}$, $\partial d_i \in C_1(\mathcal{B}_c)$ for $i \in \{1,\ldots,n\}$, represent a basis of $H_{1}(\mathcal{B}_c)$.

Now, the non-local Faraday's equations are expressed as
\begin{equation}\label{nl_faraday}
\langle \tilde{\mathbf{U}},\partial d_j \rangle = \langle - i \omega \tilde{\boldsymbol{\Phi}},d_j \rangle, j \in \{1,\ldots,\beta_1(\mathcal{K}_a)\}.
\end{equation}
A novel way to express the $j$th non local Faraday law in term of the unknowns is to pre-multiplying by $\mathbf{c}^{j\,T}$:\footnote{This correspond to algebraically sum local Faraday's equations enforced on dual faces belonging to the support of the considered cut. Since the contributions in the interior cancel out, what remains is the non local Faraday's law enforced on the boundary of the considered cut, for more details see \cite{cmame}. The cochains on primal complex are denoted by column vectors $\mathbf{c}^i$. For the sake of parsimony in the notation, the chains $d_i$ on the dual complex dual to $\mathbf{c}^i$ are denoted by $\mathbf{c}^{iT}$.}
$$\mathbf{c}^{j\,T} \left(\delta \tilde{\mathbf{U}} + i \omega \tilde{\boldsymbol{\Phi}}\right)=0,$$
and using the same passages as when obtaining (\ref{eq1prime}) we get
\begin{equation}\label{nl_faraday2}
\left(i \omega \mathbf{c}^{j\,T} \boldsymbol{\mu} \delta \right) \boldsymbol{\Omega} + \left(\mathbf{c}^{j\,T} \mathbf{R}\right) \mathbf{T} + \sum_{j=1}^{\beta_1(\mathcal{K}_a)} \left(\mathbf{c}^{j\,T} \mathbf{R} \mathbf{c}^j\right) i_j=0.
\end{equation}

By multiplying the (\ref{eq1prime}) by $i \omega$ and considering also the equations (\ref{eq2prime}) and (\ref{nl_faraday2}), the following final symmetric algebraic system having $\left.\mathbf{T}\right|_{\mathcal{K}_c\setminus\mathcal{K}_a}$, $\boldsymbol{\Omega}$ and the $\{i_j\}_{j=1}^{\beta_1(\mathcal{K}_a)}$ as unknowns reads as
\begin{equation}\label{final}
\begin{array}{lr}
i \omega \tilde{\delta} \boldsymbol{\mu} \delta \boldsymbol{\Omega}
+
i \omega \tilde{\delta} \boldsymbol{\mu} \mathbf{T}
+
 \sum_{j=1}^{\beta_1(\mathcal{K}_a)} i \omega \tilde{\delta} \boldsymbol{\mu} \mathbf{c}^j\,i_j  = \mathbf{0},&\\
\langle i \omega \boldsymbol{\mu} \delta \boldsymbol{\Omega} + \mathbf{R T} + \sum_{j=1}^{\beta_1(\mathcal{K}_a)} \mathbf{R} \mathbf{c}^j i_j, f_{\mathcal{B}} \rangle = 0,& \forall f_{\mathcal{B}} \in \mathcal{B}_c \setminus \mathcal{B}_a,\\
\left(i \omega \mathbf{c}^{j\,T} \boldsymbol{\mu} \delta \right) \boldsymbol{\Omega} + \left(\mathbf{c}^{j\,T} \mathbf{R}\right) \mathbf{T} + \sum_{j=1}^{\beta_1(\mathcal{K}_a)} \left(\mathbf{c}^{j\,T} \mathbf{R} \mathbf{c}^j\right) i_j=0, &j \in \{1,\ldots,\beta_1(\mathcal{K}_a)\},\\
\langle \mathbf{T},e \rangle =0, &\forall e \in \mathcal{K}_c\setminus \mathcal{K}_a.
\end{array}
\end{equation}
The source of the problem can be enforced by considering one of the currents $i_j$ as known, substituting it into (\ref{final}), and moving its contribution on the right-hand side of the system. Alternatively, one can force an e.m.f. by putting its value on the right-hand side of the system at the position of the correspondent non-local Faraday's equation, see \cite{sst}.

\section{A historical survey on the definitions of cuts}\label{coh}
In this Section, three families of definition of cuts presented in the literature in the last twenty-five years are reviewed. Due to the use of the Finite Elements with nodal basis functions, most definitions concentrate on the so-called \textit{thin cuts}, which are $2$-chains on the primal complex.
With the modern Finite Elements employing edge elements basis functions, cuts defined as in this paper---called \textit{thick cuts}---are needed in place of the thin cuts. Even though it is difficult to generate a thick cut from a thin cut in general, see \cite{cmame}, the definitions of thin cuts presented in the following may be easily adapted as attempts to define thick cuts also. With nodal basis functions, there is the need to impose a potential jump across the thin cuts. This is usually performed by ``cutting'' the cell complex in correspondence of the thin cuts doubling the nodes belonging to each cut. This method requires non self-intersecting thin cuts and provides some complication when cuts intersect, see for example \cite{phung}. On the contrary, the use of edge elements, as done in this paper, yields to a straightforward implementation even when cuts intersect or, as frequently happens, have self-intersections.

\subsection{Embedded sub-manifolds}
Kotiuga, starting from 1986, published many papers about the definition of cuts, proof of their existence and the development of an algorithm to compute them, see \cite{k1}, \cite{k2}, \cite{k3}, \cite{GrossKotiuga}. He defined thin cuts as embedded sub-manifolds being generators of the second relative homology group basis $H_2(\mathcal{K}_a,\partial \mathcal{K}_a)$.
He proposed also an algorithm to automatically generate cuts: first a $H^1(\mathcal{B}_a)$ basis is computed by employing a reduction technique based on a tree-cotree decomposition followed by a reordering and a classical Smith Normal Form computation \cite{Munkres}. Then, a non-physical Poisson problem is solved for each cut. Finally, cuts are extracted as iso-surfaces of the non-physical problems solutions.

This definition, although a real breakthrough, is too conservative when employing the modern edge element basis functions. In fact, in this case there is no need for cuts to be embedded surfaces, so the solution of the non-physical problems---which is computationally quite costly---can be avoided.

\subsection{Homotopy-based definition}
After Kotiuga's definition, many researchers were persuaded about the existence of an easier and more intuitive way to tackle this problem. In our opinion, two reasons diverted researchers on heuristic solutions: the first is due to the fact that algebraic topology, namely cohomology theory, was---and probably still is---not well known to scientists working in computational electromagnetics. The second---perhaps even more important---was the lack of efficient algorithms for the computation in a reasonable time of cohomology generators.

In \cite{bott}, Bott and Tu stated ``By some divine justice the homotopy groups or a finite polyhedron or a manifold seem as difficult to compute as they are easy to define.'' In fact, a homotopy-based definition of cuts has been introduced becoming soon the most popular one. The idea is to introduce a set of $2$-cells whose removal transform the insulating region into a connected and simply-connected one \cite{simkin1}, \cite{binns1}, \cite{rodger}, \cite{simkin2}, \cite{dular}\footnote{This is an attempt for a homotopy-based definition of thin cuts. Ren in \cite{ren} modifies this definition for the thick cuts. A thick cut is defined as a set of $3$-cells whose removal transform the insulating region into a simply-connected one.}.
Nonetheless, when dealing with homotopy, one falls easily into intractable problems, with a consequent lack of rigorous proofs and details of the algorithms in all the cited papers.

It has been already shown, for example in \cite{binns2}, \cite{binns3}, that there exist cuts that do not fulfil the homotopy-based definition. Namely, in case of a knot's complement, the cut realization as embedded sub-manifold---which is a Seifert surface---has to leave the complement multiply-connected. Even though this counter-example was quite clear, dealing with knotted conductors is extremely uncommon in practice. It has been concluded that problems in the homotopy-based definition arise only when dealing with knot's complement which, as written explicitly by Bossavit \cite[p. 238]{bossa_libro}, are really marginal in practice.

Nonetheless, computational electromagnetics community seems not to be aware that problems do happen frequently even with the most simple example possible, namely a conducting solid torus in which the current flow. In fact, in this paper we present for the first time a concrete counter-example that the homotopy-based definition of cuts is not only too restrictive but wrong. What is even more serious is that it makes very hard even to detect if a potential cut is correct or not, which clearly shows that this heuristic definition of cuts should be abandoned.

\begin{figure} [!h]
\centering
\includegraphics[width=10cm]{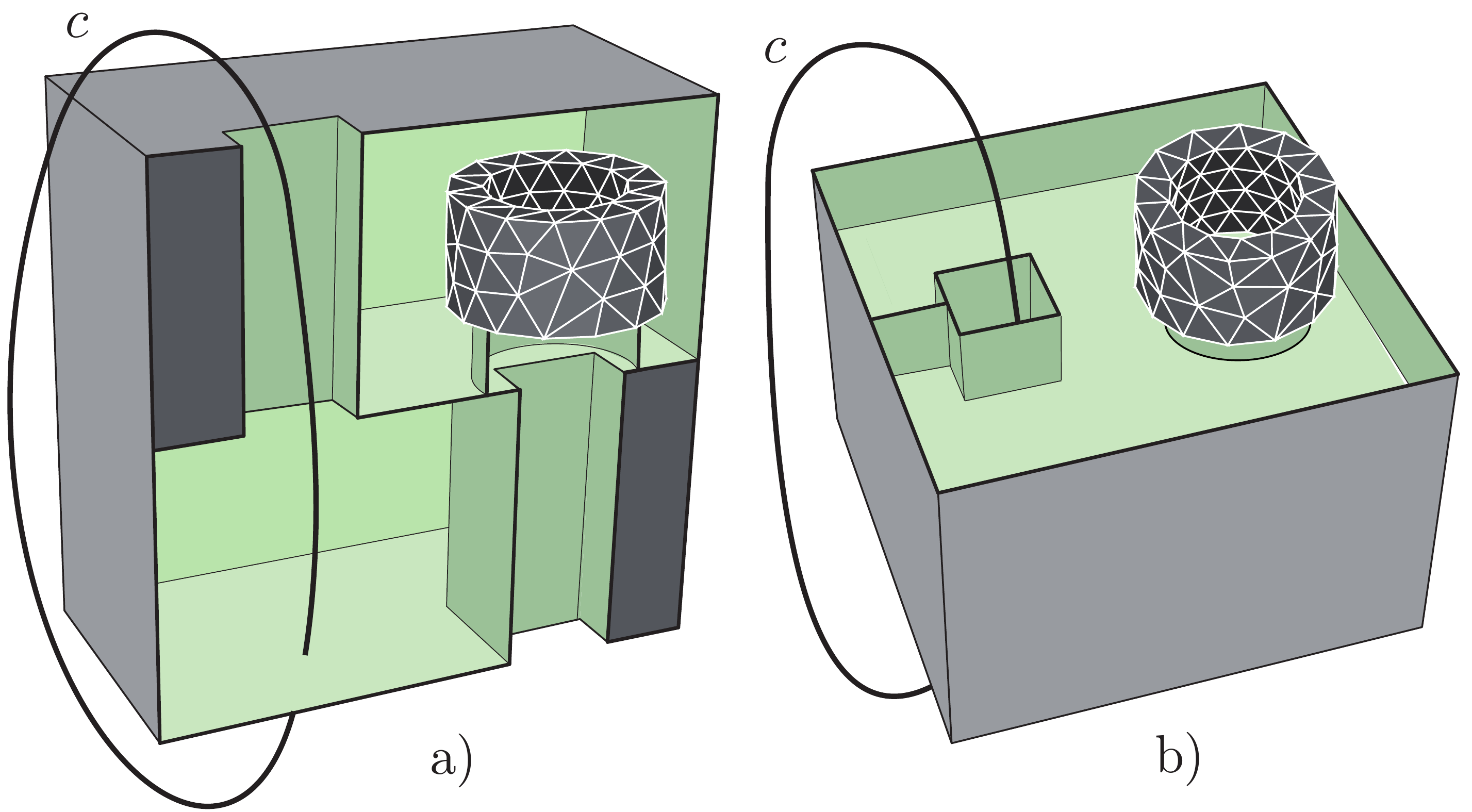}
\caption{The Bing's house counter-example.} \label{fig:bing}
\end{figure}

The counter-example is as follows. Consider a solid torus, which represents $\mathcal{K}_c$, and its complement $\mathcal{K}_a$ with respect to a ball, which contains the solid torus. By growing an acyclic sub-complex and taking the complement, the set of $2$-cells in Fig. \ref{fig:bing} is obtained. Two views are shown in the picture by cutting the set of $2$-cells with a vertical plane (Fig. \ref{fig:bing}a) and a horizontal plane (Fig. \ref{fig:bing}b). The triangulated torus represents the $\mathcal{K}_c$ complex (the torus is not cut with the planes in the picture for the sake of clarity. The ball which contains $\mathcal{K}_c$ is not shown either for the same reason.).
The set of $2$-cells is formed by a Bing's house \cite{bing} plus a cylinder. Informally, the torus $\mathcal{K}_c$ is placed in the `upper chamber' of the Bing's house and the hole of the torus is connected to the tunnel used to enter the upper chamber by the cylinder.

It is standard to see that once one removes the set of $2$-cells from the complex $\mathcal{K}_a$, what remains becomes connected and simply-connected.
Hence, this set of $2$-cells fulfil the homotopy-based definition of cut, but of course this is not a cut since the $1$-cycle $c$ crosses one time the cut without linking any current. Moreover, it is very hard in practice to detect this situation, which makes this definition---and related algorithms---not suitable even with the simplest example.

\subsection{Axiomatic definition}
An axiomatic definition of cuts is frequently used in mathematical papers, see for example~\cite{s1},\cite{s2},\cite{s3},\cite{s4},\cite{s5},\cite{s6},\cite{s7},\cite{s8},\cite{s9},\cite{s10},\cite{s11},\\
\cite{s12}. Cuts are defined as $2$-manifolds with boundary $\{\Sigma_j\}_{j=1}^{n}$ which fulfil the following set of axioms:
\begin{itemize}
 \item The boundary of $\Sigma_j$ is located at the boundary of the meshed region in which the cuts are searched for,
 \item $\Sigma_j \cap \Sigma_i = \emptyset$ for $i \not = j$,
 \item $\Omega \setminus \bigcup_{i=1}^{n} \Sigma_i$ is pseudo-Lipschitz and simply-connected.
\end{itemize}
Once such a set of cuts is provided, the results presented in those papers can be used.
However, the presented definition of cuts is too restrictive for practical applications. One can easily see that even in case of two chained conductors it is not possible to find a set of cuts for which $\Sigma_j \cap \Sigma_i = \emptyset$ holds. The same holds for many practical configurations, for example, electric transformers.

Moreover, this axiomatic definition does not point to an algorithm to find cuts automatically, which is fundamental for practical problems since it is practically impossible to define cuts `by hand' for serious problems.

\section{Cohomology computation}
\label{sec:numericalEx}
A detailed survey on the state-of-the-art algorithms to compute cohomology group generators used in electromagnetic modeling can be found in~\cite{cmes}.
The solution proposed in this paper is to change the available codes for computing homology group generators (see~\cite{capd}, \cite{chomp}) to compute the cohomology group generators. The necessary changes, described in detail in~\cite{cmes}, are very easy to implement.

In order to obtain a computationally efficient code, a so-called \textit{shaving} procedure for cohomology has to be applied.
A \textit{reduction} in (co)homology is a procedure of removing from the complex some cells in such a way that the (co)homology groups of the complex remains unchanged. Then, the classical Smith Normal Form \cite{Munkres} computation with hyper-cubical computational complexity can be performed on the reduced complex.
A \textit{shaving} is a reduction of the complex such that the representatives of generators in the reduced complex are also representatives of generators in the initial complex. As it is explained in~\cite{cmes}, the algorithm presented in~\cite{acyclicSubspace} is a shaving for cohomology computations.

Due to a number of efficient reduction techniques used (namely, \cite{acyclicSubspace} followed by \cite{kms}), in all the tested cases the Smith Normal Form computation has been not used at all. The state-of-the-art is the implementation of the acyclic sub-complex shaving with look-up tables~\cite{cmes}, whose computational complexity is linear and is able to reduce almost always the complex down to its cohomology generators.

An example of cut generated for the complement of a trefoil knot-shaped conductor is presented in Fig. \ref{fig:trefoil}.
\begin{figure} [!h]
\centering
\includegraphics[width=12cm]{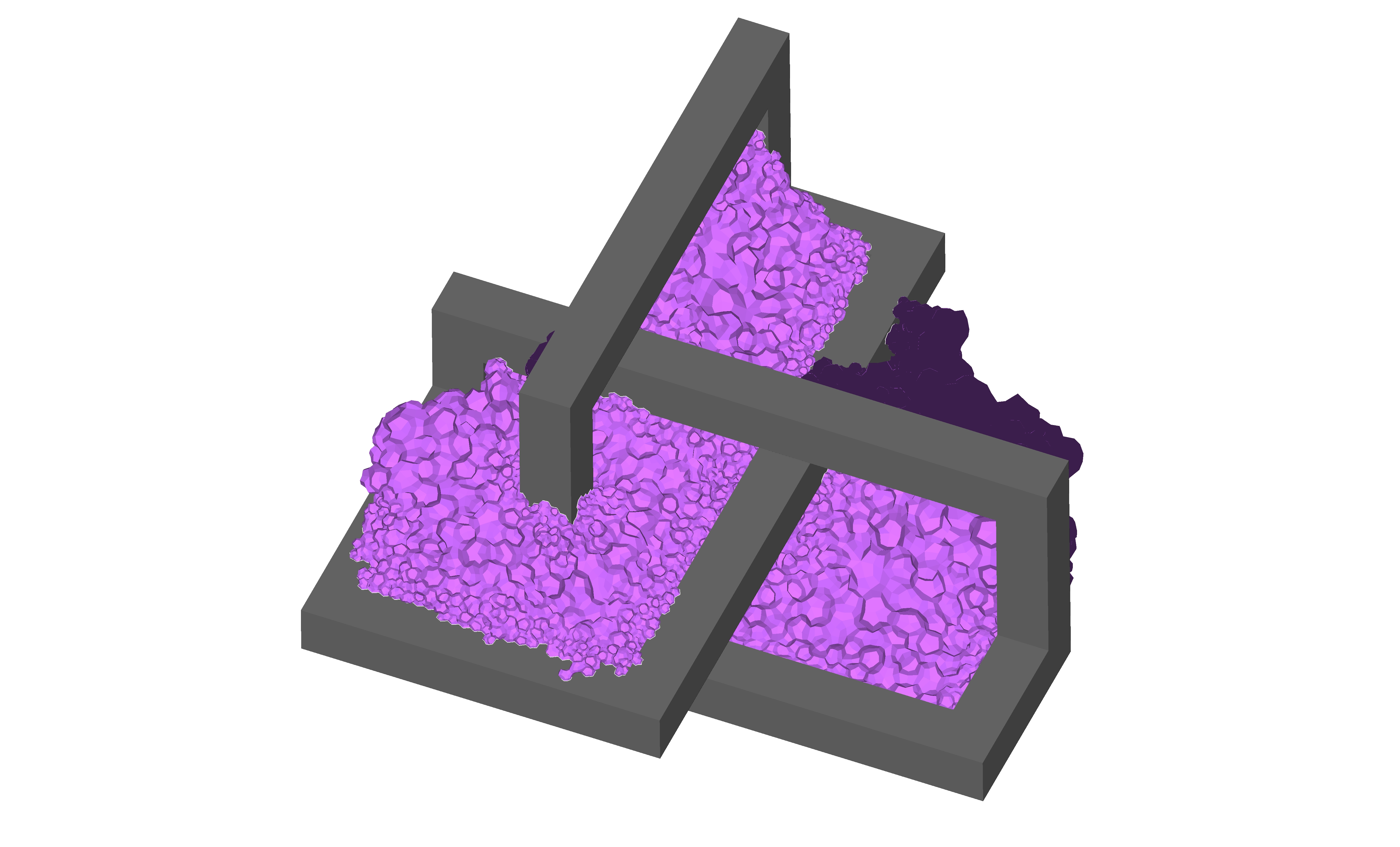}
\caption{A trefoil knot conductor together with the dual faces dual to edges belonging to the support of the cut.}
\label{fig:trefoil}
\end{figure}

We would like to point out that it is necessary for the potential design to compute the 1st cohomology group generators over integers and this cannot be substituted by any field $\mathbb{Z}_p$ for $p$ prime. In fact, let us consider $\mathbb{Z}_2$ as an example. In Fig. \ref{fig:z2}a, a two turn conductor is shown. In Fig. \ref{fig:z2}b, the support of a $2$-chain dual to a representative of a $H^1(\mathcal{K}_a,\mathbb{Z}_2)$ generator is presented. When a cycle surrounding the two branches of the conductor is considered, it does not intersect the support of the chain. It is easy to verify that on this cycle the Amp\`{e}re's law does not hold. Similar examples can be constructed for any coefficient field $\mathbb{Z}_p$.
\begin{figure} [!h]
\centering
\includegraphics[width=10cm]{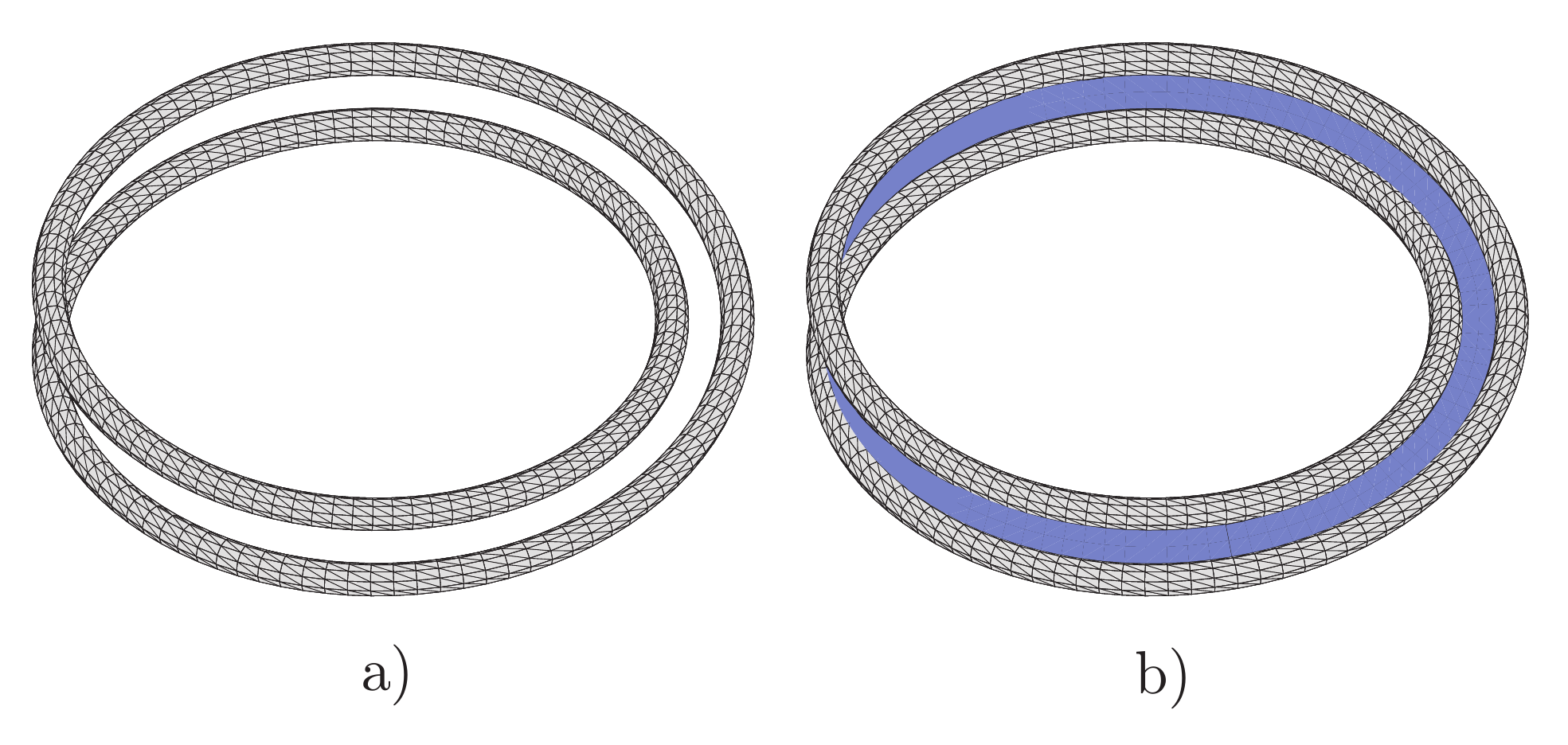}
\caption{a) A two turn conductor. b) The support of a $2$-chain dual to a representative of a $H^1(\mathcal{K}_a,\mathbb{Z}_2)$ generator.} \label{fig:z2}
\end{figure}

\section{Conclusions}
\label{sec:conclusions}
In this paper, is has been discussed how the (co)homology theories are fundamental for the potential design in computational physics.
In particular, a systematic design of potentials employed in the magneto-quasi-static $T$-$\Omega$ formulation has been presented.
It has been demonstrated that the entities called \textit{cuts} in computational electromagnetics are a basis of the first cohomology group over integers of the insulating region. The limitations on the definition of cuts presented in the literature are shown by using concrete counter-examples, which should persuade the Readers that cohomology is not one of the possible options but something which is expressly needed to the potential design.

\section*{Acknowledgments}
The Authors would like to thank Prof. P.R. Kotiuga for many useful discussions and suggestions.

\end{document}